\documentclass{article}

\usepackage{arxiv}
\usepackage[utf8]{inputenc} 
\usepackage[T1]{fontenc}    
\usepackage{hyperref}       
\usepackage{url}            
\usepackage{booktabs}       
\usepackage{amsfonts}       
\usepackage{nicefrac}       
\usepackage{microtype}      
\usepackage{lipsum}
\usepackage{amsmath,amssymb}
\usepackage{amsthm}
\usepackage{graphicx}
\usepackage{epstopdf}
\usepackage{listings}
\usepackage{color}
\usepackage{fancyhdr}
\usepackage[framemethod=tikz]{mdframed}
\usepackage{subfig}
\usepackage[nottoc,notlot,notlof]{tocbibind}
\usepackage{csvsimple}
\usepackage{longtable}
\usepackage{colortbl}
\usepackage{float}
\usepackage[]{algorithm2e}
\usepackage{verbatim}
\usepackage{lscape}
\usepackage{multirow}
\usepackage{hhline}

\newcommand{\R}{\mathbb{R}}
\newcommand{\N}{\mathbb{N}}

\DeclareMathOperator{\diam}{diam}

\DeclareMathOperator{\capop}{cap}

\DeclareMathOperator{\cl}{cl}

\def\restriction#1#2{\mathchoice
              {\setbox1\hbox{${\displaystyle #1}_{\scriptstyle #2}$}
              \restrictionaux{#1}{#2}}
              {\setbox1\hbox{${\textstyle #1}_{\scriptstyle #2}$}
              \restrictionaux{#1}{#2}}
              {\setbox1\hbox{${\scriptstyle #1}_{\scriptscriptstyle #2}$}
              \restrictionaux{#1}{#2}}
              {\setbox1\hbox{${\scriptscriptstyle #1}_{\scriptscriptstyle #2}$}
              \restrictionaux{#1}{#2}}}
\def\restrictionaux#1#2{{#1\,\smash{\vrule height .8\ht1 depth .85\dp1}}_{\,#2}}

\mdfdefinestyle{thm}{
	nobreak=true,
	hidealllines=true,
	leftline=true,
	innerleftmargin=10pt,
	innerrightmargin=10pt,
	innertopmargin=0pt,
}

\newtheorem{definition}{Definition}[section]\surroundwithmdframed[style=thm]{definition}
\newtheorem{proposition}{Proposition}[section]\surroundwithmdframed[style=thm]{proposition}
\newtheorem{theorem}{Theorem}[section]\surroundwithmdframed[style=thm]{theorem}
\newtheorem{lemma}{Lemma}[section]\surroundwithmdframed[style=thm]{lemma}
\surroundwithmdframed[style=thm]{property}
\newtheorem{problem}{Problem}[section]\surroundwithmdframed[style=thm]{problem}
\newtheorem*{problem*}{Problem}\surroundwithmdframed[style=thm]{problem*}

\theoremstyle{remark}

\newtheorem*{note}{\textbf{Remark}}

\title{Optimal Interpolation Data for PDE-based Compression of Images with Noise}

\author{
  Zakaria BELHACHMI \\ 
  IRIMAS\\
  Université de Haute-Alsace\\
  Mulhouse, France \\
  \texttt{zakaria.belhachmi@uha.fr} \\
   \And
  Thomas JACUMIN \\
  IRIMAS \\
  Université de Haute-Alsace\\
  Mulhouse, France \\
  \texttt{thomas.jacumin@uha.fr} \\
}

\begin{document}
\maketitle

\begin{abstract}
We introduce and discuss shape-based models for finding the best interpolation data in the compression of images with noise. The aim is to reconstruct missing regions by means of minimizing a data fitting term in the $L^2$-norm between the images and their reconstructed counterparts using time-dependent PDE inpainting. We analyze the proposed models in the framework of the $\Gamma$-convergence from two different points of view. First, we consider a continuous stationary PDE model, obtained by focusing on the first iteration of the discretized time-dependent PDE, and get pointwise information on the ``relevance'' of each pixel by a topological asymptotic method. Second, we introduce a finite dimensional setting of the continuous model based on ``fat pixels'' (balls with positive radius), and we study by $\Gamma$-convergence the asymptotics when the radius vanishes. 
Numerical computations are presented that confirm the usefulness of our theoretical findings 
for
non-stationary PDE-based image compression.
\end{abstract}


\section*{Introduction}

The aim of PDE-based compression is to reconstruct a given image, by inpainting 
from a set of few ``relevant pixels'', denoted by $K$, with a suitable partial differential operator. The compression is a two steps process which consists of coding part, that is the choice of the set $K$, then the decoding phase where the image is entirely recovered. Therefore, it appears intuitively that a balance between the choice quality of the set $K$, with respect to constraints such as its ``size'', as small as possible, and the location of its pixels on one hand, and the achievable accuracy of the reconstructed image, is a major key to success of the PDE-based compression. 
We quote a picture from \cite{Schmaltz2014} which expresses nicely this idea
: ``PDE-based data compression suffers from poverty, but enjoys liberty \cite{Bae2010, Belhachmi2009, Galic2008, Schmaltz2009} : Unlike in pure inpainting research \cite{Masnou1998, Bertalmio2000}, one has an extremely tight pixel
budget for reconstructing some given image. However, one is free to choose where and
how one spends this budget''. Besides that, any image compression approach should take into account the  
 nature of the considered images (e.g., noisy, textured, cartoons) and measure its impact on the selection of $K$.

The goal of the present article is to optimize the choice of such sets $K$ and to obtain, as far as possible, an analytic criteria to build it, in the spirit of \cite{Belhachmi2009}, but when the images are noisy. Optimizing over sets is a well-known field in shape optimization analysis, and many advanced theories and analytic works have been developed for various kinds of constraints on shapes and on differential operators. Our approach fits under this general  framework and show the deep links between this field and the  mathematical image analysis.

We emphasize that a comprehensive and satisfactory treatment of PDE-based compression must include both the choice of the pixels, the grey (or color) values stored and that of the inpainting operator. Actually, we know from several previous works that (e.g., \cite{Galic2005, Tschumperle2005, Bornemann2007, Galic2008, Belhachmi2009, Schmaltz2009, Bae2010}):

\begin{itemize}
\item Optimal sets, seen as optimal shapes, are not exhaustive with respect to all constraints that might be suited for image compression (e.g., easy storage, sparsity).
\item The stability of an ``optimal'' set with respect to some perturbations, when it holds, is rather weak, as it requires topologies of convergence of sets. In particular, this stability is under investigated for the case of noisy data or when some stored values are changed.
\item An optimal set, in the sens of optimal shape, is highly dependent on the inpainting operator, whereas a ``good'' operator may compensate a sub-optimal choice of pixels.
\end{itemize}

Nevertheless, finding an analytic optimal set remains in our opinion a very reasonable objective to enforce PDE-based compression methods. 

\subsection*{Related works}

Several works, notably in the field of PDE image compression, were undertaken to optimize the choice of pixels to store in the coding phase to ensure high reconstruction quality with as few as possible selected points, we refer the reader to \cite{Bae2010} and the references therein. In particular, 
in \cite{Belhachmi2009} the authors studied the choices of ``the best set'' of pixels as finding an optimal shape minimizing the semi-norm $H^1$ between the reconstructed solution and the initial noiseless image. They obtain such optimal shape in the framework of $\Gamma$-convergence approach and they give an analytic expression to build it from topological asymptotics. In \cite{Hoeltgen2015}, the authors introduced a mix of probabilistic and PDE based approach to deal with both finding optimal pixels and tonal data for discrete homogeneous harmonic inpainting. Loosely summarized, they start with a data sparsification step which consists of selecting randomly a set of pixels, then they  correct this choice within an iterative procedure which consists of a nonlocal exchange of pixels. Lastly, they optimize the grey values at these inpainting points by a least squares minimization method. This procedure is more complete than the first one as it consider both the choice of pixels and the grey values. Notice that for a fixed set $K$, the harmonic inpainting is an elliptic problem and small perturbation of the data (grey values) leads to a small perturbation on the reconstructed solution, thus, optimizing the selection of the set $K$ appears more critical for the final outcome. Whereas, a small perturbation of the sets is only ``weakly stable'' (in the sense of $\gamma$-convergence of sequences of sets, see \cite{Belhachmi2009}). 
Therefore, it seems reasonable to seek a more general problem of finding an optimal set with some stability properties. 

In this paper, we consider a shape-based analysis taking into account noisy data. We study and analyze the problem of finding a fixed set $K$ for the time harmonic linear diffusion, extending this way the approach of \cite{Belhachmi2009}. We obtain some selection criteria which are suited to the noise level. 
We compare different methods proposed and existing in the related literature in the presence of noise.

Let us now give a mathematical formulation of the problem considered. Let 
$D\subset\R^2$ the support of an image (say a rectangle) and $f : D\longrightarrow \R$,  an image which is assumed to be known only on some
region $K\subset D$. There are several PDE models to interpolate $f$ and give an approximation of the missing data. One of the basic ways is to approach $f\vert_{D\setminus K}$ by the solution of the heat equation,
having the Dirichlet boundary data $f\vert_K$ on $K$ and homogeneous Neumann boundary conditions on $\partial D$, i.e. to solve \\

\begin{problem} For $t>0$, find $u(t,\cdot)$ in $H^1(D)$ such that
	\begin{equation}
		\left\{\begin{array}{rl}
			\partial_t u(t,\cdot) - \Delta u(t,\cdot) = 0, & \text{in}\ D\setminus K, \\
			u(t,\cdot) = f, & \text{in}\ K, \\
			\frac{\partial u(t,\cdot)}{\partial \mathbf{n}} = 0, & \text{on}\ \partial D, \\
		\end{array}\right .\label{eq:linear_diffusion_filter_t}
	\end{equation}
	\[ u(0,\cdot) = u_0,\ \text{in}\ D. \]
	\label{pb:linear_diffusion_filter_t}
\end{problem}
We assume given $f\in H^1(D)$ and $\Delta f\in L^2(D)$ with $\frac{\partial f}{ \partial\mathbf{n}}=0$,  for simplicity though in practice $f$ is a function of bounded variations with a non trivial jump set. In fact, the whole analysis in the paper extends to the case of $f\in L^2$. 

To ensure the compatibility conditions with the non-homogeneous ``boundary'' conditions, we take $u(0,.)=f$ in $D$. Thus we may rewrite the problem with $v(t,x)=u(t,x)-f(x)$
\begin{problem} For $t>0$, find $v(t,\cdot)$ in $H^1(D)$ such that
	\begin{equation}
		\left\{\begin{array}{rl}
			\partial_t v(t,\cdot) - \Delta v(t,\cdot) = \Delta f, & \text{in}\ D\setminus K, \\
			v(t,\cdot) = 0, & \text{in}\ K, \\
			\frac{\partial v(t,\cdot)}{\partial \mathbf{n}} = 0, & \text{on}\ \partial D, \\
		\end{array}\right .\label{eq:linear_diffusion_filter}
	\end{equation}
	\[ v(0,\cdot) = 0,\ \text{in}\ D. \]
	\label{pb:linear_diffusion_filter}
\end{problem}

Denoting by $v_K=u_K-f$ the solution of Problem \ref{pb:linear_diffusion_filter}, the question is to identify the region $K$ which gives the “best” approximation $u_K$, in a suitable sense, for example which minimizes some $L^p$ or Sobolev norms, e.g. in \cite{Belhachmi2009}
\[ \int_D \vert\nabla u_K-\nabla f\vert^2\ dx, \]
(associated to a harmonic interpolation of $f$ in $D\setminus K$). As we want to take into account noisy images, and at the same time to perform the inpainting with denoising, a better choice a priori is to minimize the $L^p$-norms of $u_K-f$ and its gradient, particularly for $p=1$ and $p=2$,  known to be good filters for a large class of noises.
In this article, we restrict ourselves to linear time harmonic reconstruction, thus we only consider the $L^p$-norm, $p=1$ or $p=2$ for the data term. 
The choice of the set $K$, that is to say the coding part, being performed at the first step, We associate a semi-implicit discrete system to solve Problem \ref{pb:linear_diffusion_filter}. Omitting the indices $n\in\N$ and looking for the set $K$ at the first iteration, with the initial condition $v_0 = u_0-f=0$ in $D$, we are led to consider the elliptic equation :\\

\begin{problem} Find $u$ in $H^1(D)$ such that
	\begin{equation}
		\left\{\begin{array}{rl}
			u - \alpha \Delta u = f, & \text{in}\ D\setminus K, \\
			u = f, & \text{in}\ K, \\
			\frac{\partial u}{\partial \mathbf{n}} = 0, & \text{on}\ \partial D. \\
		\end{array}\right .\label{eq:problem_1}
	\end{equation}
	\label{pb:problem_1}
\end{problem}

for $\alpha=\delta t$, the time step, or equivalently \\

\begin{problem} Find $v$ in $H^1(D)$ such that
	\begin{equation}
		\left\{\begin{array}{rl}
			v - \alpha \Delta v = \alpha \Delta f, & \text{in}\ D\setminus K, \\
			v = 0, & \text{in}\ K, \\
			\frac{\partial v}{\partial \mathbf{n}} = 0, & \text{on}\ \partial D. \\
		\end{array}\right .\label{eq:problem_1:v}
	\end{equation}
	\label{pb:problem_1:v}
\end{problem}

Thus, finding the set of pixels which gives the best approximation $u_K$ in the $L^2$ sense is a shape analysis problem for the state equation of Problem \ref{pb:linear_diffusion_filter} (i.e. Problem \ref{pb:problem_1}). We recall that if $u_K$ is the solution of Problem \ref{pb:problem_1}, then $u_K$ is the minimizer of

\[ \min_{u\in H^1(D),\ u=f\ \text{in}\ K} \frac{1}{2}\int_D (u-f)^2\ dx + \frac{\alpha}{2}\int_D |\nabla (u-f)|^2\ dx-\alpha\int_D \Delta f\, (u-f)\ dx, \]

which is equivalent to

\[ \min_{u\in H^1(D),\ u=f\ \text{in}\ K} \frac{1}{2}\int_D (u-f)^2\ dx + \frac{\alpha}{2}\int_D |\nabla u|^2\ dx. \]

Following \cite{Belhachmi2009}, we develop two directions of finding that optimal set. The first is to set a continuous PDE model and search pointwise information by a topological asymptotic method. The second direction is to simulate in the continuous frame a finite dimensional shape optimization problem by imposing $K$ to be the union of a finite number of ``fat pixels''. Performing the asymptotic analysis by $\Gamma$-convergence when the number of pixels is increasing (in the same time that the fatness vanishes), we obtain useful information about the optimal distribution of the best interpolation pixels.

\subsection*{Organization of the article}

In Section \ref{sec:problem_1:continuous_model}, we introduce a mathematical model of the compression problem and its relaxed formulation. In Section \ref{sec:problem_1:topo_grad}, we compute the topological gradient of our minimization problem in order to find a mathematical criterion to construct our set of interpolation points. In Section \ref{sec:problem_1:optimal_distrib}, we change our point of view, by considering ``fat pixels'' instead of a general set of interpolation points. Finally, in Section \ref{sec:problem_1:numerical_results}, we expose some numerical results.

\section{The Continuous Model}

\label{sec:problem_1:continuous_model}

\subsection{Min-max formulation}

Let $D$ be a bounded open subset of $\R^2$. We consider the shape optimization problem 

\begin{equation}
	\min_{K\subseteq D,m(K)\leq c}\{ \mathcal{E}_p(u_K)\ |\ u_K\ \text{solution of Problem}\ \ref{pb:problem_1} \},
	\label{pb:problem_1_opt_no_constraint}
\end{equation}

where  $\mathcal{E}_p$, is defined by
\begin{equation} \label{eq:alpha_error}
	\mathcal{E}_p(u) = \frac{1}{p}\int_D |u-f|^p\ dx + \frac{\alpha}{2}\int_D |\nabla (u-f)|^2\ dx,\ \forall u\in H^1(D)\cap L^p(D),
\end{equation}  
$m$ is a measure, to be chosen, and $c>0$.
We notice that \eqref{eq:alpha_error} as a cost functional corresponds to the $L^p$ data fitting term with Tikhonov regularization \cite{Tikhonov1977}. Hence $\mathcal{E}_p$ is the simplest, and widely used, denoising PDE models at least for the values $p=1$ or $p=2$. The image compression problem aims to find an optimal set of pixels from which an accurate reconstruction of (noisy) image will be performed. Actually, the data term does not affect the $\Gamma$-convergence analysis, so we only consider the case $p=2$ and we drop the index $2$ by denoting $\mathcal{E}$ the energy. Thanks to the proposition \ref{sum} below, the analysis of the continuous model is similar to the $H^1$-semi norm case in \cite{Belhachmi2009}, we give the main analysis result and the main steps of the proof in the next section.

Let $u_K$ be the solution of Problem \ref{pb:problem_1}, it is straightforward to obtain \\
	

%



\begin{proposition}
	\label{prop:problem_1_opt_no_constraint:reformulation}
	The optimization problem \eqref{pb:problem_1_opt_no_constraint} is equivalent to 
	\[ \max_{K\subseteq D,m(K)\leq c} \min_{u\in H^1(D), u=f\ \text{in}\ K} \frac{\alpha}{2}\int_D |\nabla u|^2\ dx + \frac{1}{2} \int_D (u-f)^2\ dx. \]
\end{proposition}
Finally, problem \eqref{pb:problem_1_opt_no_constraint} can be rewritten under the unconstrained form as follows :

	\begin{equation}
		\label{pb:problem_1_opt_penalized}
		\max_{K\subseteq D} \min_{u\in H^1(D), u=f\ \text{in}\ K} \frac{\alpha}{2}\int_D |\nabla u|^2\ dx + \frac{1}{2} \int_D (u-f)^2\ dx - \beta m(K),
	\end{equation}

	for $\beta >0$.


The well-posedness of \eqref{pb:problem_1_opt_no_constraint} depends of the choice of the measure $m$. In \cite{Belhachmi2009}, it has been proven that, in the Laplacian case, choosing the $\nu$-capacity as measure $m$ leads to the existence of a relaxed formulation and the well-posedness of this optimization problem. Consequently, we will study \eqref{pb:problem_1_opt_no_constraint} when $m$ is the $\nu$-capacity. The next section is devoted to the analysis within the $\gamma$-convergence (see Appendix \ref{appendix:gamma_convergence}) approach follows the same lines as in \cite{Belhachmi2009} with slight changes.


\subsection{Analysis of the model}

The optimization problem \eqref{pb:problem_1_opt_no_constraint} can be rewritten by penalizing the Dirichlet boundary condition $u=f$ in $K$

\[ \max_{K\subseteq D} \min_{u\in H^1(D)} \frac{\alpha}{2}\int_D |\nabla u|^2\ dx + \frac{1}{2} \int_D (u-f)^2\ dx + \frac{1}{2}\int_D(u-f)^2\ d\infty_K - \beta\capop_\nu(K), \]

It is well known that such shape optimization problems do not always have a solution (e.g. \cite{Belhachmi2009}), we seek a relaxed formulation, which under the capacity constraint yields a relaxed solution, that is to say, a capacity measure. Thus, we consider the problem 

\[ \max_{\mu\in\mathcal{M}_0(D)} \min_{u\in H^1(D)} \frac{\alpha}{2}\int_D |\nabla u|^2\ dx + \frac{1}{2} \int_D (u-f)^2\ dx + \frac{1}{2}\int_D (u-f)^2\ d\mu - \beta\capop_\nu(\mu), \]

where $\mu$ is in $\mathcal{M}_0(D)$. As the $L^2$- norm is continuous, referring to Proposition \ref{sum}, we may drop from the following $\Gamma$-convergence analysis, the term 
\[
\frac{1}{2}\int_D (u-f)^2\ dx.
\]
For every $\mu$ in $\mathcal{M}_0(D)$ and $u$ in $H^1(D)$, we define $F_\mu$, from $H^1(D)$ into $\R\cup\{+\infty\}$, by

\[ F_\mu(u) := \begin{cases}
	\alpha\int_D |\nabla u|^2\ dx + \int_D (u-f)^2\ d\mu &,\ \text{if}\ |u| \leq |f|_\infty, \\
	+\infty &,\ \text{otherwise.}
\end{cases} \]

We have that $F_\mu$ is equi-coercive with respect to $\mu$, for any $\mu$ in $\mathcal{M}_0(D)$. Indeed, let $u$ be in $H^1(D)$ such that $|u| \leq |f|_\infty$, we have

\[ F_\mu(u) \geq \alpha\int_D |\nabla u|^2\ dx - 2 |f|_\infty^2 \mu(D). \]

For every $\mu$ in $\mathcal{M}_0(D)$, we define $E$, from $\mathcal{M}_0(D)$ into $\R$, by

\[ E(\mu) := \min_{u\in H^1(D)} F_\mu(u). 
\]

For a given $\mu$ in $\mathcal{M}_0(D)$, $E(\mu)$ corresponds to the energy of \\

\begin{problem} Find $u$ in $H^1(D)$ such that
	\begin{equation}
		\left\{\begin{array}{rl}
			- \alpha \Delta u  + \mu (u-f) = 0, & \text{in}\ D, \\
			\frac{\partial u}{\partial \mathbf{n}} = 0, & \text{on}\ \partial D. \\
		\end{array}\right .\label{eq::dirichlet_penalization:relaxed}
	\end{equation}
	\label{pb:problem_1:dirichlet_penalization:relaxed}
\end{problem}\vspace{0.2cm}

Thus, if $u$ is a solution of Problem \ref{pb:problem_1:dirichlet_penalization:relaxed} for a given $\mu$ in $\mathcal{M}_0(D)$, then $F_\mu(u)= E(\mu)$ and the function $u$ satisfies the maximum principle $|u| \leq |f|_\infty$. Since in the next sections we want to include balls centered at points $x_0$ in $D$, that we do not want to be too close to the boundary of $D$, we introduce the following notations for $\delta>0$,

\[ D^{-\delta} := \{ x\in D\ |\ d(x,\partial D) \geq \delta \}\subseteq D, \]

\[ \mathcal{K}^\delta(D) := \{ K\subseteq D\ |\ K\ \text{closed},\ K\subseteq D^{-\delta} \}, \]

and

\[ \mathcal{M}_0^\delta(D) := \{ \mu\in\mathcal{M}_0(D)\ |\ \restriction{\mu}{D\setminus D^{-\delta}}=0 \} \subseteq\mathcal{M}_0(D). \]

Let us consider the problem

\[ \max_{\mu\in\mathcal{M}_0^\delta(D)} \min_{u\in H^1(D)} \frac{\alpha}{2}\int_D |\nabla u|^2\ dx + \frac{1}{2}\int_D (u-f)^2\ d\mu - \beta\capop(\mu). \]

Using the compactness of $\mathcal{M}_0(D)$ for the $\gamma$-convergence (Proposition \ref{prop:problem_1:M0:gamma_compacity}) and the locality of the $\gamma(F)$-convergence (Proposition \ref{prop:problem_1:gamma_locality}), we have the following result \\

\begin{proposition}[$\gamma$-compactness of $\mathcal{M}_0^\delta(D)$]
	The set $\mathcal{M}_0^\delta(D)$ defined above is compact with respect to the $\gamma$-convergence.
	\label{prop:problem_1:M0delta:gamma_compacity}
\end{proposition}

We have also the density theorem (the proof is given in Appendix \ref{appendix:proofs_analysis}) \\

\begin{theorem}
	We have \[ \cl_\gamma \mathcal{K}_\delta(D) = \mathcal{M}_0^\delta(D), \]
	i.e., $\mathcal{K}_\delta(D)$ is dense into $\mathcal{M}_0^\delta(D)$ with respect to the $\gamma(F)$-convergence.
	\label{thm:density_of_K_delta}
\end{theorem}


Similarly to Lemma 3.4 in \cite{Belhachmi2009}, we have \\

\begin{theorem}
	Let $\mu_n\in\mathcal{K}_\delta(D)$. If $\mu_n$ $\gamma$-converge to $\mu$, then $\capop_\nu(\mu_n)\to\capop_\nu(\mu)$.
\end{theorem}

\begin{theorem}
	\label{thm:problem_1:convergence_capacity}
	If $(\mu_n)_n$ in $\mathcal{M}_0^\delta(D)$ $\gamma$-converges to $\mu$, then $\mu$ is in $\mathcal{M}_0^\delta(D)$ and $F_{\mu_n}$ $\Gamma$-converges to $F_\mu$ in $L^2(D)$.
\end{theorem}

The proof of the last theorem is also given in Appendix \ref{appendix:proofs_analysis}. Finally, we can state the main result of this section (proof in Appendix \ref{appendix:proofs_analysis}). \\

\begin{theorem}
	We have \[ \sup_{K\in\mathcal{K}_\delta(D)} \big( E(\infty_K) - \beta\capop_\nu(\infty_K) \big) = \max_{\mu\in\mathcal{M}_0^\delta(D)}\big( E(\mu) - \beta\capop_\nu(\mu) \big). \]
	\label{thm:relaxed_problem}
\end{theorem}
Replacing $F_n$ with $F_n+G$, $G:=\frac{1}{2}\int_D (u-f)^2\ dx$ and with Proposition \ref{sum}, we get the existence of an optimal solution to the relaxed formulation. 

\begin{note}
In order to solve the relaxed problem 

\[ \min_{u\in H^1(D)} \alpha\int_D |\nabla u|^2\ dx + \int_D (u-f)^2\ dx + \int_D (u-f)^2\ d\infty_K - \beta\capop_\nu(K), \]

we may use a shape derivative with respect to the measures $\mu$. However, such a method yields diffuse measures, thus too thick sets whereas we seek discrete sets of pixels. 
\end{note}

In the next two sections, we aim to find an explicit characterization of the set $K$ using topological asymptotic.
\section{Topological Gradient}
\label{sec:problem_1:topo_grad}

Here, we aim to compute the solution of our optimization problem \eqref{pb:problem_1_opt_no_constraint} by using a topological gradient-based algorithm as in \cite{Larnier2012, Garreau2001}. This kind of algorithm consists in starting with $K = \bar{D}$ and determining how making small holes in $K$ affect the cost functional to find the balls which have the most decreasing effect. To this end, let us define $K_\varepsilon$ the compact set $K\setminus B(x_0,\varepsilon)$ where $B(x_0,\varepsilon)$ is the ball centered in $x_0\in D$ with radius $\varepsilon>0$ such that $B(x_0,\varepsilon)\subset K$. From now, we consider 
the functional : 
\[ j^* : A\subset D \mapsto \min_{u\in H^1(D), u=f\ \text{in}\ A} \frac{1}{2} \int_D (u-f)^2\ dx + \frac{\alpha}{2}\int_D |\nabla u|^2\ dx, \]
or equivalently,

\[ j : A\subset D \mapsto \min_{v\in H^1(D), v=0\ \text{in}\ A} \frac{1}{2} \int_D v^2\ dx + \frac{\alpha}{2}\int_D |\nabla v|^2\ dx - \int_D g v\ dx, \]

where $g:= \alpha\Delta f$. Finally, we denote by $v_\varepsilon$ the minimizer of $j(K_\varepsilon)$. Then, we have \\

\begin{proposition} With notations from above, we have when $\varepsilon$ tends to $0$,
	\[ j(K_\varepsilon) - j(K) = \frac{\pi}{2}\big(g(x_0)\big)^2\varepsilon^2\ln(\varepsilon) + O(\varepsilon^2). \]
	\label{prop:topologicalGradient}
\end{proposition}
\begin{proof}
    \begin{align*}
        j(K_\varepsilon) - j(K) &= \frac{\alpha}{2}\int_{B(x_0,\varepsilon)} |\nabla v_\varepsilon|^2\ dx + \frac{1}{2} \int_{B(x_0,\varepsilon)} v_\varepsilon^2\ dx - \int_{B(x_0,\varepsilon)} g v_\varepsilon\ dx.
    \end{align*}

	The weak formulation of Problem \ref{pb:problem_1} leads to 
	\begin{align*}
		j(K_\varepsilon) - j(K) &= \frac{1}{2}\int_{B(x_0,\varepsilon)} g v_\varepsilon\ dx - \int_{B(x_0,\varepsilon)} g v_\varepsilon\ dx \\
		&= -\frac{1}{2}\int_{B(x_0,\varepsilon)} g v_\varepsilon\ dx.
	\end{align*}
	
    We have $g(x) = g(x_0) + \Vert x - x_0\Vert O(1)$, and hence
    
    \[ j(K_\varepsilon) - j(K)= -\frac{1}{2}g(x_0)\int_{B(x_0,\varepsilon)} v_\varepsilon\ dx + \varepsilon\,O(1)\int_{B(x_0,\varepsilon)} v_\varepsilon\ dx. \]
    
    It is enough to compute the fundamental term in the asymptotic development of the expression $\int_{B(x_0,\varepsilon)} v_\varepsilon\ dx$. This is done by using Proposition \ref{prop:appendix:int_calculus}.
\end{proof}

    Since for $\varepsilon < 1$, $\ln\varepsilon < 0$, the result above suggests to keep the points $x_0$ where $ |\Delta f(x_0)|^2$ is maximal, when $\varepsilon$ small enough. From a practical point of view, this is the main result of our local shape analysis. In the next section, we will see that such a strict threshold rule might be relaxed. 

\section{Optimal Distribution of Pixels : The ``Fat Pixels'' Approach}
\label{sec:problem_1:optimal_distrib}

In this section, we change our point of view by considering ``fat pixels'' instead of a general set of interpolation points. In the sequel, we will follow \cite{Belhachmi2009, Buttazzo2006}. We restrict our class of admissible sets as an union of balls which represent pixels. For $m>0$ and $n\in\N$, we define

\[ \mathcal{A}_{m,n} := \Big\{ \overline{D}\cap\bigcup_{i=1}^n\overline{B(x_i,r)}\ \Big|\ x_i\in D_r,\ r=mn^{-1/2} \Big\}, \]

where $D_r$ is the $r$-neighborhood of $D$. The following analysis remains unchanged in $\mathbb{R}^d$, but for the sake of simplicity we restrict ourselves to the case $d=2$. We consider problem \eqref{pb:problem_1_opt_no_constraint} for every $K\in\mathcal{A}_{m,n}$ i.e. 

\begin{equation*}
	\min_{K\in\mathcal{A}_{m,n}}\Big\{ \frac{1}{2} \int_D (u_K-f)^2\ dx + \frac{\alpha}{2}\int_D |\nabla u_K-\nabla f|^2\ dx\ \Big|\ u_K\ \text{solution of Problem}\ \ref{pb:problem_1} \Big\}.
\end{equation*}

Like in the previous section, we set $v_K := u_K- f$. This last optimization problem can be reformulated as a compliance optimization problem :

\begin{equation}
	\min_{K\in\mathcal{A}_{m,n}}\Big\{ \frac{1}{2} \int_D g\, v_{K_n}\ dx\ \Big|\ u_K:= v_K+f\ \text{solution of Problem}\ \ref{pb:problem_1} \Big\},
	\label{pb:problem_1_opt_no_constraint_pixel}
\end{equation}

where $g := \alpha\Delta f$, like in the previous sections. Here, we do not need to specify a size constraint on our admissible domains. Indeed, imposing $K\in\mathcal{A}_{m,n}$ implies a volume constraint and a geometrical constraint on $K$ since $K$ is formed by a finite number of balls with radius $mn^{-1/2}$. We deal with Neumann boundary conditions on $D$. However, it is possible to cover the boundary with $\frac{2C_{D}}{m}n^{1/2}$ balls so that we have formally homogeneous Dirichlet boundary conditions on $D$. The well-posedness of such a problem has been studied in the Laplacian case in \cite{Buttazzo2006}. Without significant change we have \\

\begin{theorem}
	If $D$ is an open bounded subset of $\R^2$ and if $g\geq 0$ is in $L^2(D)$, then the problem \eqref{pb:problem_1_opt_no_constraint_pixel} admits a unique solution.
\end{theorem}

If we denote by $K_n^\text{opt}$ the solution, then we have that $\infty_{K_n^\text{opt}}$ $\gamma$-converge to $\infty_D$ as $n$ tends to $+\infty$. However, the number of pixels $x_0$ in $D$ to keep goes also to infinity. Thus, it gives no relevant information on the distribution of the points to retain. As pointed out in \cite{Bucur2005}, the local density of $K_n^\text{opt}$ can be obtained by using a different topology for the $\Gamma$-convergence of the rescaled energies. In this new frame, the minimizers are unchanged but their behavior is seen from a different point of view. We define the probability measure $\mu_K$ for a given set $K$ in $\mathcal{A}_{m,n}$ by

\[ \mu_K := \frac{1}{n}\sum_{i=1}^n \delta_{x_i}. \]

We define the functional $F_n$ from $\mathcal{P}(\bar D)$ into $[0,+\infty]$ by

\[ F_n(\mu) := \begin{cases} n\int_D gv_K\ dx &,\ \text{if}\ \exists K\in\mathcal{A}_{m,n},\ \text{s.t.}\ \mu=\mu_K, \\
+\infty&,\ \text{otherwise}. \end{cases} \]

\vspace{0.2cm}

The following $\Gamma$-convergence of $F_n$ theorem is similar to the one given in Theorem 2.2. in \cite{Buttazzo2006}. \\

\begin{theorem}
    \label{thm:g-convergence}
	If $g\geq 0$, then the sequence of functionals $F_n$, defined above, $\Gamma$-converge with respect to the weak $\star$ topology in $\mathcal{P}(\bar{D})$ to 
	\[ F(\mu) := \int_D \frac{g^2}{\mu_a}\theta(m\mu_a^{1/2})\ dx, \]
	where $\mu = \mu_a dx + \nu$ is the Radon--Nikodym--Lebesgue decomposition of $\mu$ (\cite{Folland2013}, Theorem 3.8) with respect to the Lebesgue measure and
	\[ \theta(m) := \inf_{K_n\in\mathcal{A}_{m,n}} \liminf_{n\to +\infty} n \int_D gv_{K_n}\ dx, \]
	$v_{K_n} := u_{K_n} - f$, $u_{K_n}$ solution of Problem \ref{pb:problem_1}.
\end{theorem}

As a consequence of the $\Gamma$-convergence stated in the theorem above, the empirical measure $\mu_{K_n^\text{opt}} \to \mu^\text{opt}$ weak $\star$ in $\mathcal{P}(\R^d)$ where $\mu^\text{opt}$ is a minimizer of $F$. Unfortunately, the function $\theta$ is not known explicitly. We establish here after that $\theta$ is positive, non-increasing and vanish after some point which will be enough for practical exploration. The next theorem gives an estimate of the function $\theta$ defined above. The proof is given in Appendix \ref{appendix:theta}. \\

\begin{theorem}
	We have, for $m$ in $(0,t_1)$, \[ C_1(\alpha)|\ln(m)| - C_2(\alpha) \leq \theta(m) \leq C_3(\alpha)|\ln(m)|, \]

	where $C_1$, $C_2$ and $C_3$ are constants depending on $\alpha$.
\end{theorem}

\begin{note}
    We can extend the results above to any $g$ since we may formally split the discussion on the sets $\{ g \geq 0 \}$ and $\{ g < 0 \}$.
\end{note} \vspace{0.2cm}

These estimates on $\theta$ suggest that to minimize $F$, when $|g|$ is large, $\mu_a$ should be large in order for $\theta$ to be close to its vanishing point, while when $|g|$ is small $\mu_a$ could be small. Formal Euler-Lagrange equation and the estimates on $\theta$ give the following information : to minimize 

\[ F(\mu) := \int_D \frac{g^2}{\mu_a}\theta(m\mu_a^{1/2})\ dx, \]

one have to take

\[ \frac{\mu_a^2}{|1-\log\mu_a|} \approx c_{m,f}\, g^2. \]

This introduces a soft-thresholding with respect to the first approach. To sum up, we can choose the interpolation data such that the pixel density is increasing with $|g| = |\Delta f|$. This soft-thresholding rule can be enforced with a standard digital halftoning. According to \cite{Belhachmi2009, Ulichney1987, Adler2003}, digital halftoning is a method of rendering that convert a continuous image to a binary image, for example black and white image, while giving the illusion of color continuity. This color continuity is simulated for the human eye by a spacial distribution of black and white pixels. Two different kinds of halftoning algorithms exist : dithering and error diffusion halftoning. The first one is based on a so-called dithering mask function, while the other one is an algorithm which propagate the error between the new value ($0$ or $1$) and the old one (in the interval $[0,1]$). An ideal digital halftoning method would conserves the average value of gray while giving the illusion of color continuity.
\section{Numerical Results}
\label{sec:problem_1:numerical_results}

In this section, we present some numerical simulations to validate the previous theoretical analysis and we compare to other commonly used methods of image compression. We discretize the PDEs with a standard implicit finite difference scheme on a quasi-uniform mesh in order to make the comparisons easy. We have considered the method presented in this article that we will denote by  $L^2$-methods, more precisely we call \textit{L2-T} the algorithm based on hard thresholding with the criteria obtained in Section \ref{sec:problem_1:topo_grad}, \textit{L2-H} the algorithm based on the fat pixels variant (soft thresholding) and each algorithm is used with, respectively without, the halftoning based on  Floyd-Steinberg dithering algorithm \cite{Floyd1976}. The methods that we use for comparison purposes are 
the B-Tree algorithm \cite{Distasi1997} and a random mask selection. Next we discuss and present some extensions of the method in several ways : first, we allow a data modification on the compression set $K$ to test how under the same framework and analysis the selected masks may be eventually improved. Secondly, we consider images corrupted with Salt and Pepper noise, though the $L^2$-norms based reconstruction are less efficient. Finally, we consider the case of color images. We will denote $f$ the initial image, $f_\delta$ its noisy version and $u$ the reconstructed one.


\subsection{Numerical simulations and Comparisons}

For the $L^2$-methods, respectively, $H^1$ based methods, we implement the hard threshold criteria, namely we select the pixels where  $\vert\Delta f\vert$ is maximum and the soft threshold algorithm of the fat pixels approach, where the selected pixels are chosen according to the distribution of $|\Delta f|$. The last algorithm uses a dithering procedure \cite{Floyd1976}.

In Table \ref{tab:methods-comparison:0.05}, Table \ref{tab:methods-comparison:0.1} and Table \ref{tab:methods-comparison:0.15}, we give the $L^2$-errors between the images $f$ and $u$, as a function of the noise level for each method. 
We notice that the $L^2$-errors are
better than with B-Tree and random choices when the noise magnitude of the data is not too high whereas it deteriorates increasingly with the noise. In fact, the locations where $\vert\Delta f\vert$ is high includes more noisy pixels which is reflected in the mask selection. This effect of taking more noisy pixels is amplified with compression ratio. We emphasize that our comparisons are only concerned with the influence of noise on the coding phase in compression and are by no means exhaustive. In particular, when the noise level is too high the criterion based on the locations where the Laplacian is maximum appears less efficient with respect to B-Tree (which include by construction an amount of denoising) or even the random choice of pixels, we will see how to improve the criterion in these cases.

In Figure \ref{fig:methods-comparison:0.1:wn:0}, Figure \ref{fig:methods-comparison:0.1:wn:0.03} and Figure \ref{fig:methods-comparison:0.1:wn:0.05}, we present various masks obtained and the corresponding reconstructed images. We notice that with no noise or low level ones, the masks consist of pixels located on, and close to, the edges which is intuitively expected. The soft threshold method includes few pixels from the homogeneous areas leading to a better reconstruction results. As the noise magnitude grows, increasingly noisy pixels are selected in the mask  leading to poor reconstructions.   

\begin{table}[H]
    \centering
    \begin{tabular}{|c|c|c|c|c|c|}
        \hline
        \multirow{2}{*}{\textbf{Noise}} & \multicolumn{1}{c|}{\textbf{L2-T}} & \multicolumn{1}{c|}{\textbf{L2-H}} & \multicolumn{2}{c|}{\textbf{B-tree}} & \textbf{Rand} \\
        \cline{2-6}
              & $\|f-u\|_2$ & $\|f-u\|_2$  & $\alpha$ & $\|f-u\|_2$ & $\|f-u\|_2$ \\
        \hhline{|======|}
           0  & 39.17 &  9.56 & 10000 &  9.88 & 15.02 \\
        \hline
         0.03 & 13.47 & 12.14 &   30  & 10.61 & 15.49 \\
        \hline
         0.05 & 17.10 & 15.12 &   70  & 11.80 & 16.13 \\
        \hline
         0.1  & 31.43 & 23.98 &   87  & 15.50 & 19.01 \\
        \hline
         0.2  & 75.48 & 41.92 &   68  & 24.87 & 27.64 \\
        \hline
    \end{tabular} \vspace{0.2cm}
    
    \caption{$L^2$-error between the original image $f$ and the reconstruction $u$ (build from  $f_\delta$)  with $5\%$ of total pixels saved.}
    \label{tab:methods-comparison:0.05}
\end{table}

    

\begin{table}[H]
    \centering
    \begin{tabular}{|c|c|c|c|c|c|}
        \hline
        \multirow{2}{*}{\textbf{Noise}} & \multicolumn{1}{c|}{\textbf{L2-T}} & \multicolumn{1}{c|}{\textbf{L2-H}} & \multicolumn{2}{c|}{\textbf{B-tree}} & \textbf{Rand} \\
        \cline{2-6}
              & $\|f-u\|_2$ & $\|f-u\|_2$  & $\alpha$ & $\|f-u\|_2$ & $\|f-u\|_2$ \\
        \hhline{|======|}
           0  & 25.57 &  4.92 & 10000 &  6.44 & 10.94 \\
        \hline
         0.03 &  9.36 &  8.59 &   80  &  7.56 & 11.85 \\
        \hline
         0.05 & 13.91 & 12.66 &   62  &  9.57 & 13.14 \\
        \hline
         0.1  & 27.39 & 23.19 &   56  & 15.16 & 16.94 \\
        \hline
         0.2  & 61.46 & 43.29 &   51  & 26.73 & 27.89 \\
        \hline
    \end{tabular} \vspace{0.2cm}
    
    \caption{$L^2$-error between the original image $f$ and the reconstruction $u$ with $10\%$ of total pixels saved.}
    \label{tab:methods-comparison:0.1}
\end{table}

    

\begin{table}[H]
    \centering
    \begin{tabular}{|c|c|c|c|c|c|}
        \hline
        \multirow{2}{*}{\textbf{Noise}} & \multicolumn{1}{c|}{\textbf{L2-T}} & \multicolumn{1}{c|}{\textbf{L2-H}} & \multicolumn{2}{c|}{\textbf{B-tree}} & \textbf{Rand} \\
        \cline{2-6}
              & $\|f-u\|_2$ & $\|f-u\|_2$  & $\alpha$ & $\|f-u\|_2$ & $\|f-u\|_2$ \\
        \hhline{|======|}
           0  & 18.21 &  3.32 & 10000 &  4.55 &  9.03 \\
        \hline
         0.03 &  8.21 &  7.67 &   51  &  6.54 &  9.93 \\
        \hline
         0.05 & 13.05 & 12.07 &   15  &  8.91 & 11.35 \\
        \hline
         0.1  & 25.91 & 23.11 &   22  & 15.36 & 16.78 \\
        \hline
         0.2  & 54.66 & 43.34 &   12  & 28.20 & 28.77 \\
        \hline
    \end{tabular} \vspace{0.2cm}
    
    \caption{$L^2$-error between the original image $f$ and the reconstruction $u$ with $15\%$ of total pixels saved.}
    \label{tab:methods-comparison:0.15}
\end{table}

\begin{figure}[H]
	\centering
	\subfloat[Input image.]{
		\includegraphics[height=3.6cm]{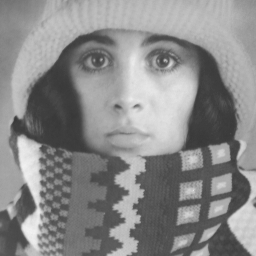}
	}
	\quad
	\subfloat[$\sigma=0.03$, $\|f-f_\delta\|_2=7.65$.]{
		\includegraphics[height=3.6cm]{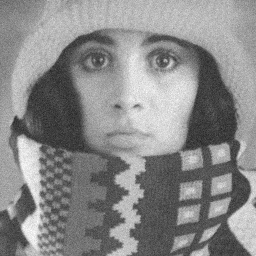}
	}
	\quad
	\subfloat[$\sigma=0.05$, $\|f-f_\delta\|_2=12.81$.]{
		\includegraphics[height=3.6cm]{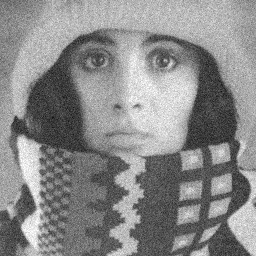}
	}
	\quad\subfloat[$\sigma=0.1$, $\|f-f_\delta\|_2=25.45$.]{
		\includegraphics[height=3.6cm]{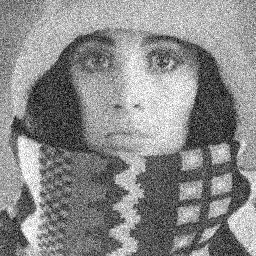}
	}
	\caption{Input images with and without gaussian noise of standard deviation $\sigma$.}
\end{figure}

\begin{figure}[H]
	\centering
	\subfloat[Mask with L2-T method.]{
		\includegraphics[height=3.6cm]{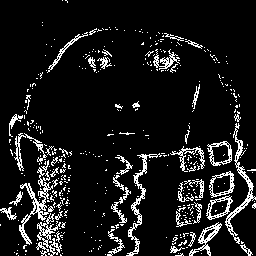}
	}
	\quad
	\subfloat[Reconstruction with L2-T method.]{
		\includegraphics[height=3.6cm]{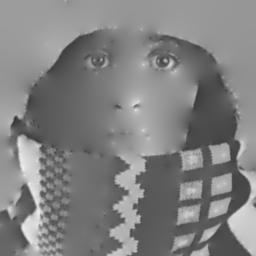}
	}
	\quad
	\subfloat[Mask with L2-H method.]{
		\includegraphics[height=3.6cm]{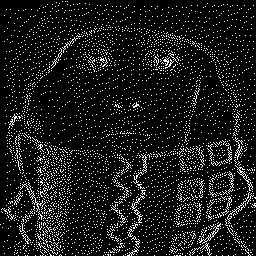}
	}
	\quad
	\subfloat[Reconstruction with L2-H method.]{
		\includegraphics[height=3.6cm]{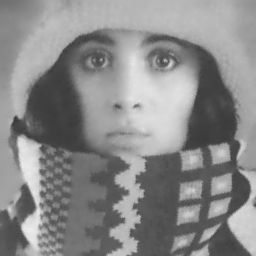}
	}
\end{figure}
\begin{figure}[H]
	\centering
	\subfloat[Mask with B-TREE method.]{
		\includegraphics[height=3.6cm]{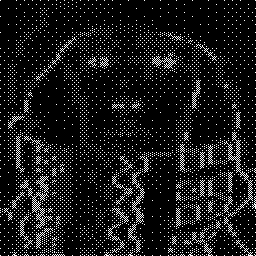}
	}
	\quad
	\subfloat[Reconstruction with B-TREE method.]{
		\includegraphics[height=3.6cm]{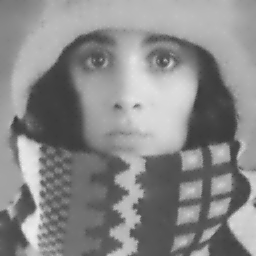}
	}
	\quad
	\subfloat[Mask with RAND method.]{
		\includegraphics[height=3.6cm]{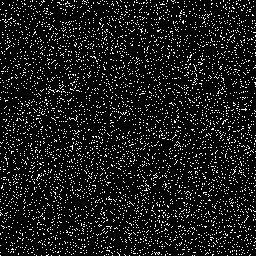}
	}
	\quad
	\subfloat[Reconstruction with RAND method.]{
		\includegraphics[height=3.6cm]{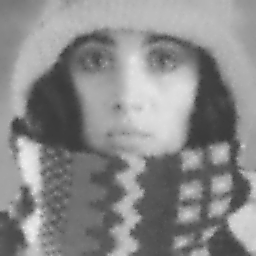}
	}
	\caption{Masks and reconstructions for Table \ref{tab:methods-comparison:0.1} when the input image is noiseless ($\sigma=0$).}
	\label{fig:methods-comparison:0.1:wn:0}
\end{figure}

\begin{figure}[H]
	\centering
	\subfloat[Mask with L2-T method.]{
		\includegraphics[height=3.6cm]{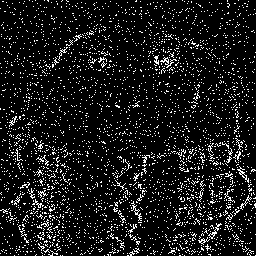}
	}
	\quad
	\subfloat[Reconstruction with L2-T method.]{
		\includegraphics[height=3.6cm]{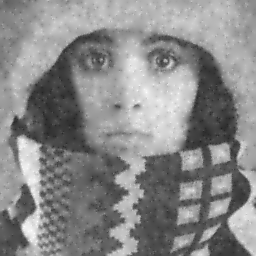}
	}
	\quad
	\subfloat[Mask with L2-H method.]{
		\includegraphics[height=3.6cm]{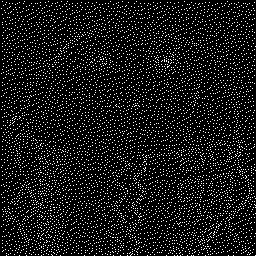}
	}
	\quad
	\subfloat[Reconstruction with L2-H method.]{
		\includegraphics[height=3.6cm]{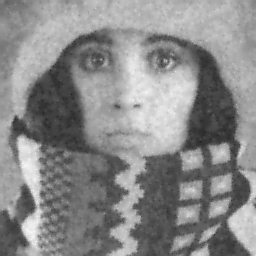}
	}
\end{figure}
\begin{figure}[H]
	\centering
	\subfloat[Mask with B-TREE method.]{
		\includegraphics[height=3.6cm]{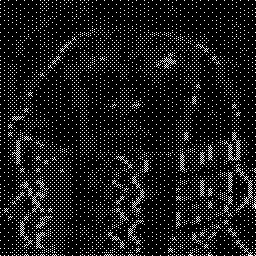}
	}
	\quad
	\subfloat[Reconstruction with B-TREE method.]{
		\includegraphics[height=3.6cm]{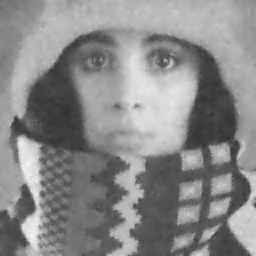}
	}
	\quad
	\subfloat[Mask with RAND method.]{
		\includegraphics[height=3.6cm]{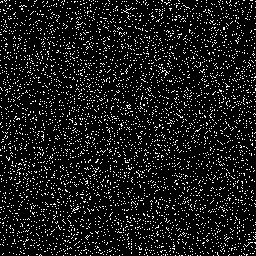}
	}
	\quad
	\subfloat[Reconstruction with RAND method.]{
		\includegraphics[height=3.6cm]{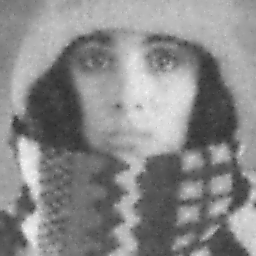}
	}
	\caption{Masks and reconstructions for Table \ref{tab:methods-comparison:0.1} when the input image is affected by gaussian noise ($\sigma=0.03$).}
	\label{fig:methods-comparison:0.1:wn:0.03}
\end{figure}

\begin{figure}[H]
	\centering
	\subfloat[Mask with L2-T method.]{
		\includegraphics[height=3.6cm]{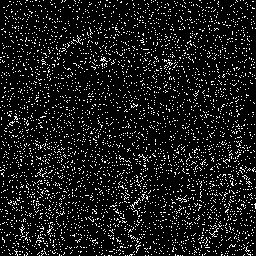}
	}
	\quad
	\subfloat[Reconstruction with L2-T method.]{
		\includegraphics[height=3.6cm]{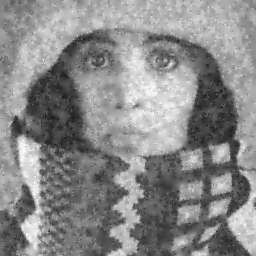}
	}
	\quad
	\subfloat[Mask with L2-H method.]{
		\includegraphics[height=3.6cm]{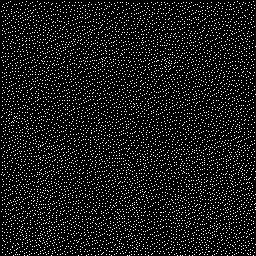}
	}
	\quad
	\subfloat[Reconstruction with L2-H method.]{
		\includegraphics[height=3.6cm]{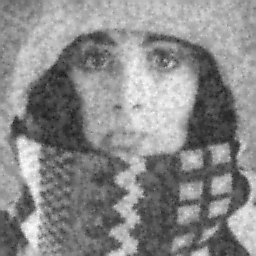}
	}
\end{figure}
\begin{figure}[H]
	\centering
	\subfloat[Mask with B-TREE method.]{
		\includegraphics[height=3.6cm]{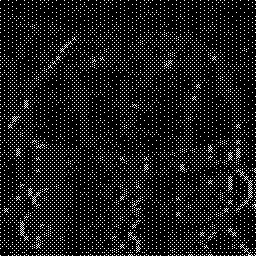}
	}
	\quad
	\subfloat[Reconstruction with B-TREE method.]{
		\includegraphics[height=3.6cm]{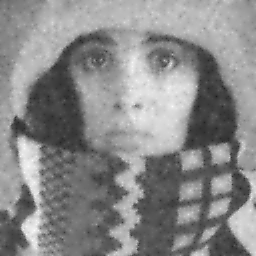}
	}
	\quad
	\subfloat[Mask with RAND method.]{
		\includegraphics[height=3.6cm]{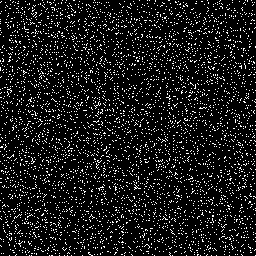}
	}
	\quad
	\subfloat[Reconstruction with RAND method.]{
		\includegraphics[height=3.6cm]{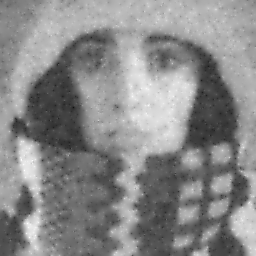}
	}
	\caption{Masks and reconstructions for Table \ref{tab:methods-comparison:0.1} when the input image is affected by gaussian noise ($\sigma=0.05$).}
	\label{fig:methods-comparison:0.1:wn:0.05}
\end{figure}

\subsection{Deeper Comparison with B-Tree}

B-Tree algorithms seem in some examples to perform slightly better in terms of the $L2$-error. Actually, this is not surprising as B-Tree approach build the masks for compression by optimizing with respect to this norm. However, this is not a disadvantage to our approach when we compare with respect to some other constraints, e.g.  
 \begin{itemize}
 \item the cost of the compression is higher than with our method in term of CPU time. This becomes worst with images of high resolution,
 \item B-Tree works only with regular grids, which is a serious limitation for images where features of importance (e.g. edges) are located outside the grid. For images with high anisotropy this shortcoming is more critical,
 \item in B-Tree algorithms refinement are necessary such as the choice of the parameters $\alpha$ which make them more image dependent.
 \end{itemize}
 Our approach gives an analytic criteria which allows to overcome most of these difficulties. We added more comparisons elements between the two approaches in the revised version to give a more complete image on their outcomes Table \ref{tab:methods-comparison-btree} and Figure \ref{fig:methods-comparison-btree}.

\begin{table}[H]
    \centering
    \begin{tabular}{|c|c|c|c|c|c|c|}
        \hline
        \multirow{2}{*}{\textbf{Noise}} & \multicolumn{2}{c|}{\textbf{L2-H}} & \multicolumn{3}{c|}{\textbf{B-tree}} \\
        \cline{2-6}
              & $\|f-u\|_2$ & time (s) & $\alpha$ & $\|f-u\|_2$ & time (s) \\
        \hhline{|======|}
          0  &  9.36 & 0.37 & 10000 & 15.65 & 198.23 \\
        \hline
         0.03 & 12.19 & 0.38 & 19.59 & 16.16 & 33.92 \\
        \hline
         0.05 & 15.91 & 0.38 & 12.20 & 16.42 & 20.58 \\
        \hline
         0.1  & 23.26 & 0.39 & 16.57 & 17.80 & 13.71 \\
        \hline
         0.2  & 36.30 & 0.37 & 10.00 & 21.85 & 20.70 \\
        \hline
    \end{tabular} \vspace{0.2cm}
    
    \caption{$L^2$-error between the original image $f$ and the reconstruction $u$ (build from  $f_\delta$)  with $9\%$ of total pixels saved.}
    \label{tab:methods-comparison-btree}
\end{table}

\begin{figure}[H]
	\centering
	\subfloat[Input image.]{
		\includegraphics[width=3.6cm]{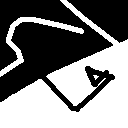}
	}
	\quad
	\subfloat[$\sigma=0.03$, $\|f-f_\delta\|_2=2.70$.]{
		\includegraphics[width=3.6cm]{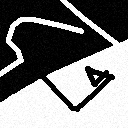}
	}
	\quad
	\subfloat[$\sigma=0.05$, $\|f-f_\delta\|_2=4.55$.]{
		\includegraphics[width=3.6cm]{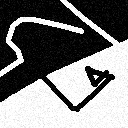}
	}
	\quad\subfloat[$\sigma=0.1$, $\|f-f_\delta\|_2=9.17$.]{
		\includegraphics[width=3.6cm]{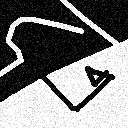}
	}
	\caption{Input images with and without gaussian noise of standard deviation $\sigma$.}
\end{figure}
\begin{figure}[H]
	\centering
	\subfloat[Mask with L2-H method.]{
		\includegraphics[width=3.6cm]{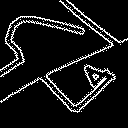}
	}
	\quad
	\subfloat[Reconstruction with L2-H method.]{
		\includegraphics[width=3.6cm]{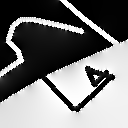}
	}
	\quad
	\subfloat[Mask with B-Tree method.]{
		\includegraphics[width=3.6cm]{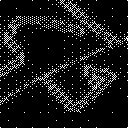}
	}
	\quad
	\subfloat[Reconstruction with B-Tree method.]{
		\includegraphics[width=3.6cm]{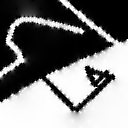}
	}
\end{figure}
\begin{figure}[H]
	\centering
	\subfloat[Mask with L2-H method.]{
		\includegraphics[width=3.6cm]{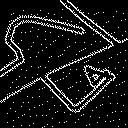}
	}
	\quad
	\subfloat[Reconstruction with L2-H method.]{
		\includegraphics[width=3.6cm]{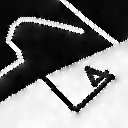}
	}
	\quad
	\subfloat[Mask with B-Tree method.]{
		\includegraphics[width=3.6cm]{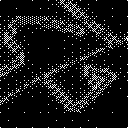}
	}
	\quad
	\subfloat[Reconstruction with B-Tree method.]{
		\includegraphics[width=3.6cm]{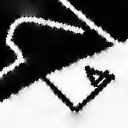}
	}
\end{figure}
\begin{figure}[H]
	\centering
	\subfloat[Mask with L2-H method.]{
		\includegraphics[width=3.6cm]{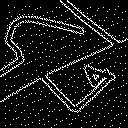}
	}
	\quad
	\subfloat[Reconstruction with L2-H method.]{
		\includegraphics[width=3.6cm]{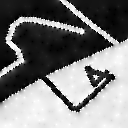}
	}
	\quad
	\subfloat[Mask with B-Tree method.]{
		\includegraphics[width=3.6cm]{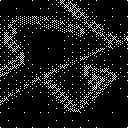}
	}
	\quad
	\subfloat[Reconstruction with B-Tree method.]{
		\includegraphics[width=3.6cm]{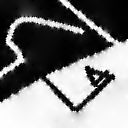}
	}
\end{figure}
\begin{figure}[H]
	\centering
	\subfloat[Mask with L2-H method.]{
		\includegraphics[width=3.6cm]{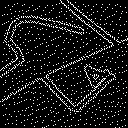}
	}
	\quad
	\subfloat[Reconstruction with L2-H method.]{
		\includegraphics[width=3.6cm]{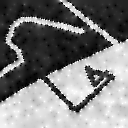}
	}
	\quad
	\subfloat[Mask with B-Tree method.]{
		\includegraphics[width=3.6cm]{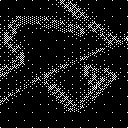}
	}
	\quad
	\subfloat[Reconstruction with B-Tree method.]{
		\includegraphics[width=3.6cm]{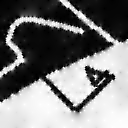}
	}
	
	\caption{Masks and reconstructions for Table \ref{tab:methods-comparison-btree}.}
	\label{fig:methods-comparison-btree}
\end{figure}

\subsection{Improving the selection criteria}

The issue raised from the above experiments and analysis is how to improve the mask selection as $\vert\Delta f\vert$ appears to be very sensitive to the noise magnitude? A way considered by \cite{Mainberger2012, Hoeltgen2015} is to resort to tonal optimization where the data is simultaneously modified on $K$. In this paper we investigate first a more basic idea on how to modify the criterion under the same analysis presented in the previous sections. 
It is clear that any change in the data on the Dirichlet condition on $K$ will causes a modification (e.g. a correction) on the final criterion. Intuitively, taking $u=g$ on $K$ where $g$ is  either less noisy than $f$ or a copy of $f$ with enhanced edges, it would lead a best pixels selection. The simplest ways to this are presented now.

\subsubsection{Sharpening the edges}

We replace $f$ by $\tilde{f} := f-\beta\Delta f$, $\beta > 0$. Then the criterion become : 

\[ \max |\Delta\tilde{f}| \Leftrightarrow \max |\Delta f-\beta\Delta^2 f|. \]

\begin{figure}[H]
	\centering
	\subfloat[$\sigma=0$.]{
		\includegraphics[height=3.6cm]{resources/images/Trui.png}
	}
	\quad
	\subfloat[$\tilde{f}$.]{
		\includegraphics[height=3.6cm]{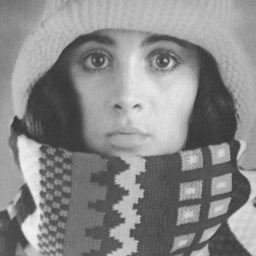}
	}
	\quad
	\subfloat[Mask with $\tilde{f}$.]{
		\includegraphics[height=3.6cm]{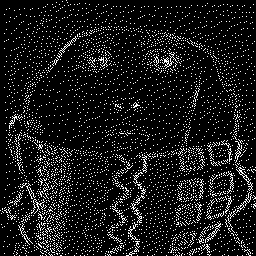}
	}
	\quad
	\subfloat[Reconstruction with $\tilde{f}$.]{
		\includegraphics[height=3.6cm]{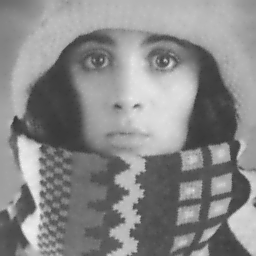}
	}
	\caption{With L2-H $\beta = 0.18$, $\|f-u\|_2=4.78$.}
\end{figure}

\begin{figure}[H]
	\centering
	\subfloat[$\sigma=0.05$.]{
		\includegraphics[height=3.6cm]{resources/images/Trui-wn-0.05.png}
	}
	\quad
	\subfloat[$\tilde{f}$.]{
		\includegraphics[height=3.6cm]{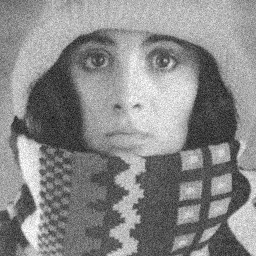}
	}
	\quad
	\subfloat[Mask with $\tilde{f}$.]{
		\includegraphics[height=3.6cm]{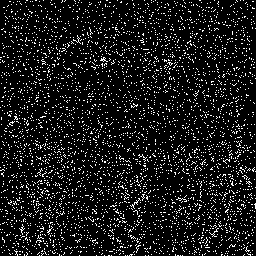}
	}
	\quad
	\subfloat[Reconstruction with $\tilde{f}$.]{
		\includegraphics[height=3.6cm]{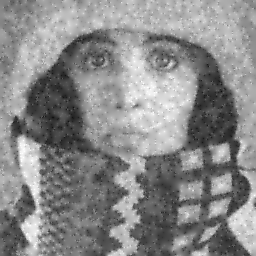}
	}
	\caption{With L2-T $\beta = 0.00005$, $\|f-u\|_2=13.91$.}
\end{figure}

\begin{figure}[H]
	\centering
	\subfloat[$\sigma=0.05$.]{
		\includegraphics[height=3.6cm]{resources/images/Trui-wn-0.05.png}
	}
	\quad
	\subfloat[$\tilde{f}$.]{
		\includegraphics[height=3.6cm]{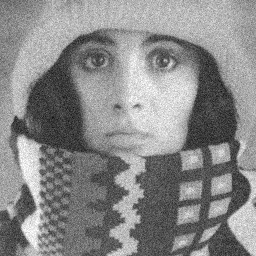}
	}
	\quad
	\subfloat[Mask with $\tilde{f}$.]{
		\includegraphics[height=3.6cm]{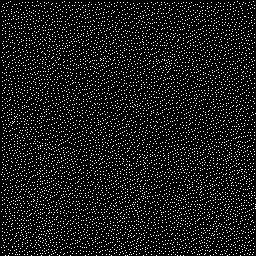}
	}
	\quad
	\subfloat[Reconstruction with $\tilde{f}$.]{
		\includegraphics[height=3.6cm]{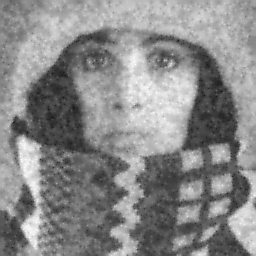}
	}
	\caption{With L2-H $\beta = 0.0009$, $\|f-u\|_2=12.57$.}
\end{figure}
We notice that sharpening the edges lead to a slight improvement of the accuracy of the reconstruction, however this effect decreases as the noise level increases. This is due to the action of the operators $\Delta$ and $\Delta^2$ which enforce the edges but also amplify the noise.

\subsubsection{(Pre)-Filtering the data}

The idea here is to perform a small amount of filtering of the initial data. This may be performed on coarse mesh with the goal of reducing slightly the noise level. Thus, we replace $f$ by $\tilde{f}\in H^1(D)$ solution of $\tilde{f} -\beta\Delta \tilde{f} = f$ in $D$. The criterion become : 

\[ \max |\Delta\tilde{f}| \Leftrightarrow \max |\tilde{f} - f|. \]
We notice a significant improvement in the accuracy and the obtained mask. It is also important to notice that there is no need of blind and strong denoising as the linear filter is applied on coarse mesh with the aim of small reduction of the noise in the homogeneous areas, where intuitively we expect few pixels in the mask, and the similar action near the edges, with even some amount of blurring which do not affect too much the selection criterion. 

We emphasize that others possible improvements within the same framework may be considered but clearly, a more systematic study in the spirit of tonal optimization is certainly a better choice in this direction.




\begin{figure}[H]
	\centering
	\subfloat[$\sigma=0.05$.]{
		\includegraphics[height=3.6cm]{resources/images/Trui-wn-0.05.png}
	}
	\quad
	\subfloat[$\tilde{f}$.]{
		\includegraphics[height=3.6cm]{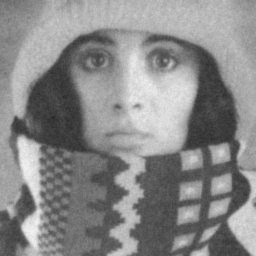}
	}
	\quad
	\subfloat[Mask with $\tilde{f}$.]{
		\includegraphics[height=3.6cm]{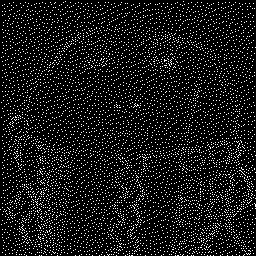}
	}
	\quad
	\subfloat[Reconstruction with $\tilde{f}$.]{
		\includegraphics[height=3.6cm]{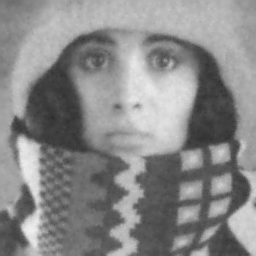}
	}
	\caption{With L2-H $\beta = 1.2$, $\|f-u\|_2=7.77$.}
\end{figure}

\subsection{Impulse Noise}

We consider now images corrupted with impulse noise. In 
Figure \ref{fig:experiments:method:saltpepper}, Figure \ref{fig:experiments:method:salt} and Figure \ref{fig:experiments:method:pepper} we make some  experiments with $2\%$ of salt and pepper noise, respectively $1\%$ of only salt noise and $1\%$ of only pepper noise. These numerical simulations show that $L^2$-methods do not give a satisfying reconstruction with this sort of noise. In fact, the Laplacian takes large values at the noisy pixels. Thus such pixels are selected in the inpainting mask, whereas linear diffusion denoising as it is well known, do not perform well (e.g. large stains in Figure \ref{fig:experiments:method:saltpepper} (c)).
This suggests to minimize the $L^1$-errors (instead of $L^2$ \cite{Nikolova2002, Nikolova2004}) to remove the impulse noise like salt and pepper noise as shown in the figures. 

\begin{figure}[H]
	\centering
	\subfloat[Original image with $2\%$ of salt and pepper noise.]{
		\includegraphics[height=3.6cm]{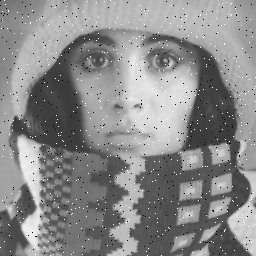}
	}
	\quad
	\subfloat[L2-H mask.]{
		\includegraphics[height=3.6cm]{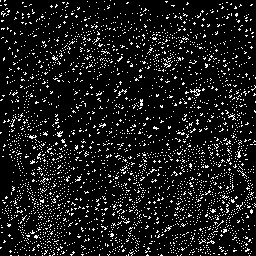}
	}
	\quad
	\subfloat[Reconstruction.]{
		\includegraphics[height=3.6cm]{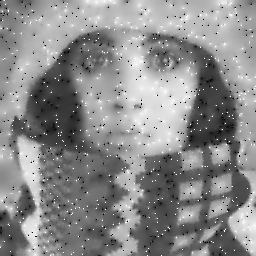}
	}
	\quad
	\subfloat[$L^1$-error minimizer.]{
		\includegraphics[height=3.6cm]{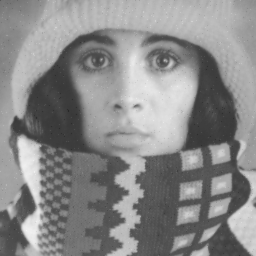}
	}
	
	\caption{Image reconstruction with $10\%$ of total pixels saved and $2\%$ of salt and pepper noise applied to the input image.}
	\label{fig:experiments:method:saltpepper}
\end{figure} 

\begin{figure}[H]
	\centering
	\subfloat[Original image with $1\%$ of salt noise.]{
		\includegraphics[height=3.6cm]{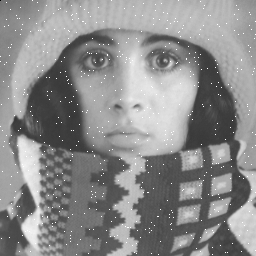}
	}
	\quad
	\subfloat[L2-H mask.]{
		\includegraphics[height=3.6cm]{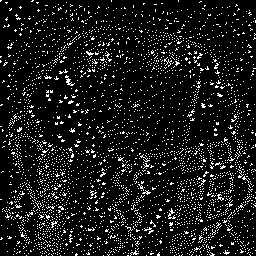}
	}
	\quad
	\subfloat[Reconstruction.]{
		\includegraphics[height=3.6cm]{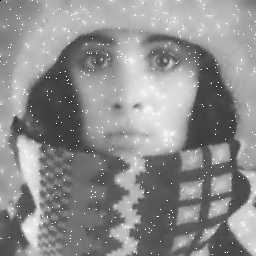}
	}
	\quad
	\subfloat[$L^1$-error minimizer.]{
		\includegraphics[height=3.6cm]{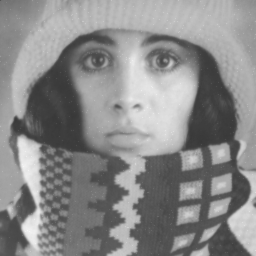}
	}
	
	\caption{Image reconstruction with $10\%$ of total pixels saved and $1\%$ of salt noise applied to the input image.}
	\label{fig:experiments:method:salt}
\end{figure} 

\begin{figure}[H]
	\centering
	\subfloat[Original image with $1\%$ of pepper noise.]{
		\includegraphics[height=3.6cm]{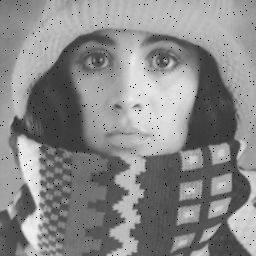}
	}
	\quad
	\subfloat[L2-H mask.]{
		\includegraphics[height=3.6cm]{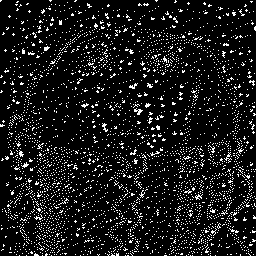}
	}
	\quad
	\subfloat[Reconstruction.]{
		\includegraphics[height=3.6cm]{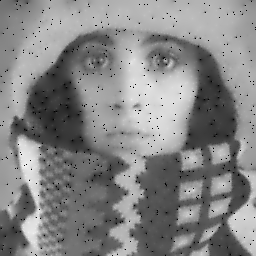}
	}
	\quad
	\subfloat[$L^1$-error minimizer.]{
		\includegraphics[height=3.6cm]{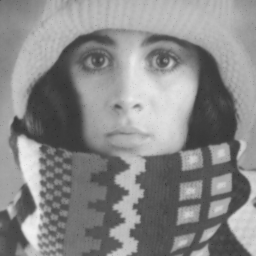}
	}
	
	\caption{Image reconstruction with $10\%$ of total pixels saved and $1\%$ of pepper noise applied to the input image.}
	\label{fig:experiments:method:pepper}
\end{figure} 

	

\subsection{Colored Images}

A colored image can be modeled by a function $f$ from $D$ to $[0,1]^3$, $x \mapsto \big(f_\text{R}(x), f_\text{G}(x), f_\text{B}(x)\big)^\text{T}$, where functions $f_\text{R}$, $f_\text{G}$ and $f_\text{B}$ are from $D$ to $\R$, represent red channel, green channel and blue channel respectively (Figure \ref{fig:experiments:method:color:input}). 
Our strategy is to create three masks, one for each channel. This is done in Figure \ref{fig:experiments:method:color} where (a) is the original image and (c) is the reconstruction by keeping $10\%$ of total pixels for each mask. Since we compute a mask with a fixed number of pixels for each channel, the final mask, where the three masks are combined, may not have the same number of pixels. In fact, the three masks may not have common pixels or only some common pixels. More efficient strategies that use YCbCr color space instead of RGB space, have been investigated in \cite{Peter2014, Peter2017, Mohideen2020}.

\begin{figure}[H]
	\centering
	\subfloat[Original image.]{
		\includegraphics[height=3.6cm]{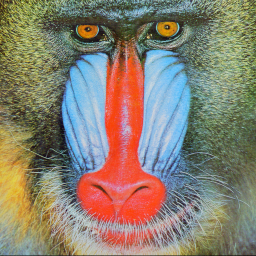}
	}
	\quad
	\subfloat[Red channel to grayscale.]{
		\includegraphics[height=3.6cm]{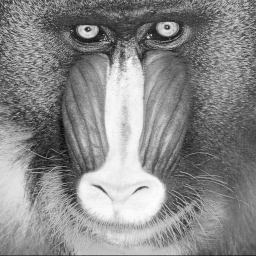}
	}
	\quad
	\subfloat[Green channel to grayscale.]{
		\includegraphics[height=3.6cm]{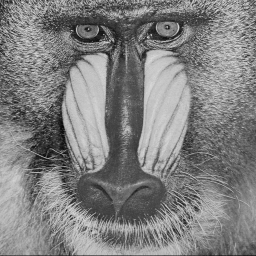}
	}
	\quad
	\subfloat[Blue channel to grayscale.]{
		\includegraphics[height=3.6cm]{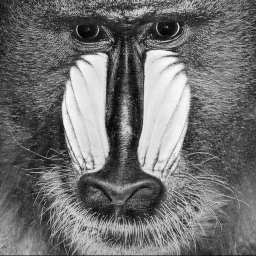}
	}

	\caption{Reconstructions for colored images.}
	\label{fig:experiments:method:color:input}
\end{figure}

\begin{figure}[H]
	\centering
	\subfloat[Mask \textit{L2-STA-T}.]{
		\includegraphics[height=3.6cm]{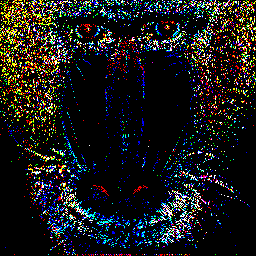}
	}
	\quad
	\subfloat[Reconstruction \textit{L2-STA-T}.]{
		\includegraphics[height=3.6cm]{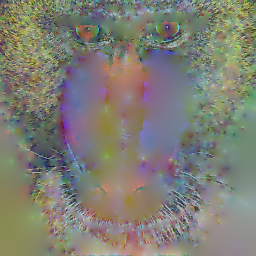}
	}
	\quad
	\subfloat[Mask \textit{L2-STA-H}.]{
		\includegraphics[height=3.6cm]{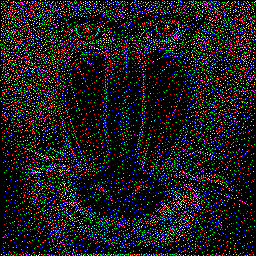}
	}
	\quad
	\subfloat[Reconstruction \textit{L2-STA-H}.]{
		\includegraphics[height=3.6cm]{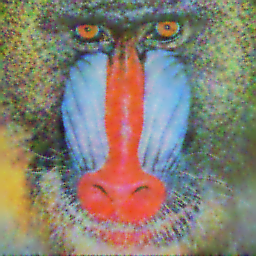}
	}

	\caption{Reconstructions for colored images.}
	\label{fig:experiments:method:color}
\end{figure}

\section*{Summary and Conclusions}
\label{sec:problem_1:conclusion}

We have considered the problem of finding the best interpolation data in PDE-based compression problems for images with noise. We aim to have a unified framework for both compression and denoising since it is not clear that doing this two tasks separately leads to satisfying results. We introduced a geometric variational model to determine a set $K$ which minimizes the $L^p$-distance between the initial image and its reconstruction from the datum in $K$. We extended the shape optimization approach introduced in \cite{Belhachmi2009} 
based on the analysis in the framework of $\Gamma$-convergence. In particular, we studied the two approaches considered there which differ in the way a single pixel in $K$ is taken. Both theoretical findings emphasize the importance of the Laplacian of appropriate data and highlight the deep connection between the geometric set and the inpainting operator (in our case time harmonic). 
We have performed several numerical tests and comparisons which demonstrate the efficiency of the approach in handling images with noise. Besides, ongoing research addresses on one hand, a systematic study of tonal optimization techniques as a further step towards a drastic reduction of the ``size'' of $K$ without loss of accuracy. Secondly, extending the shape analysis methods to nonlinear reconstruction operators will open some exciting perspectives in the fields of PDE based image compression.

This work may have several application like compression of ``real world'' images (since they always contain noise), video compression using only variational methods \cite{Andris2016}, microscopy imaging or
denoising by inpainting \cite{Adam2017}.

\bibliographystyle{siam}  
\bibliography{template}

\appendix

\newpage\section{Framework of the \texorpdfstring{$\gamma$-}-convergence}
\label{appendix:gamma_convergence}

For the sake of completeness, we recall the definition of the $\nu$-capacity, for $\nu>0$, $\Gamma$-convergence and $\gamma$-convergence written in \cite{Belhachmi2009}. More details about the $\nu$-capacity or the shape optimization tools can be found in \cite{DalMaso1993, Bucur2005}. Let us start with some definitions. \\

\begin{definition}[$\nu$-capacity of a set]
	Let $V\subseteq\R^d$ be a smooth bounded open set and $\nu>0$. We define the $\nu$-capacity of a subset $E$ in $V$ by
	\[ \capop_\nu(E,V) = \inf\Big\{ \int_V |\nabla u|^2\ dx + \nu \int_V u^2\ dx\ \Big|\ u\in H^1_0(V),\ u\geq 1 \text{ a.e. in } E \Big\}. \]
\end{definition}
\begin{note}
	We notice that if, for a given set $E\subseteq V$ and $\nu>0$, we have $\capop_\nu(E,V) = 0$, then we have $\capop_\nu(E,V) = 0$ for every $\nu >0$. Thus, the sets of vanishing capacity are the same for all $\nu >0$. That is why we will drop the $\nu$ and simply write $\capop$ instead of $\capop_\nu$.
\end{note} \vspace{0.2cm}

\begin{definition}[quasi-everywhere property]
	We say that a property holds \textbf{quasi-everywhere} if it holds for all $x$ in $E$ except for the elements of a set $Z$ subset of $E$ such that $\capop(Z,E)=0$. We write q.e.
\end{definition}

\begin{definition}[quasi-open set]
	We say that a subset $A$ of $E$ is \textbf{quasi-open} if for every $\varepsilon >0$ there exists an open subset $A_\varepsilon$ of $D$, such that $A\subseteq A_\varepsilon$ and $\capop(A_\varepsilon\setminus A, D)<\varepsilon$.
\end{definition}

We introduce the set $\mathcal{M}_0(D)$ which is denoted by  $\mathcal{M}_0^*(D)$ in \cite{DalMaso1987}. \\

\begin{definition}[The set $\mathcal{M}_0(D)$]
	We denote by $\mathcal{M}_0(D)$ the set of all non negative Borel measures $\mu$ on $D$, such that

	\begin{itemize}
		\item $\mu(B)=0$, for every Borel set $B$ subset of $D$ with $\capop(B,D) = 0$,
		\item $\mu(B)=\inf\big\{\mu(U)\ \big|\ U\ \text{quasi-open},\ B\subseteq U\big\}$, for every Borel subset $B$ of $D$.
	\end{itemize}
\end{definition}

\begin{definition}[$\nu$-capacity of a measure]
	The $\nu$-capacity, for $\nu>0$, of a measure $\mu$ of $\mathcal{M}_0(D)$ is defined by

	\[ \capop_\nu(\mu, D) := \inf_{u\in H^1_0(D)} \int_D |\nabla u|^2\ dx + \nu\int_D u^2\ dx + \int_D (u-1)^2\ d\mu. \]
\end{definition}

The next proposition gives us a natural way to identify a set to a measure of $\mathcal{M}_0(D)$. \\

\begin{proposition} Let $E$ be a Borel subset of $D$. We denote by $\infty_E$ the measure of $\mathcal{M}_0(D)$ defined by

	\[ \infty_E(B) := \begin{cases}
		+\infty &,\ \text{if}\ \capop(B\cap E,D)>0,\\
		0 &, \ \text{otherwise}.
	\end{cases},\ \text{for all}\ B\ \text{Borel subset of}\ D. \]
\end{proposition}

\begin{note}
	For a given Borel subset $E$ of $D$, we have $\capop_\nu(\infty_E,D) = \capop_\nu(E,D)$.
\end{note}\vspace{0.2cm}

\begin{definition}[$\Gamma$-convergence] 
	Let $V$ be a topological space. We say that the sequence of functionals $(F_n)_n$, from $V$ into $\bar{\R}$, $\Gamma$-converges to $F$ in $V$ if

	\begin{itemize}
		\item for every $u$ in $V$, there exists a sequence $(u_n)_n$ in $V$ such that $u_n\to u$ in $V$ and $F(u)\geq \limsup_{n\to+\infty} F_n(u_n)$,
		\item for every sequence $(u_n)_n$ in $V$ such that $u_n\to u$ in $V$, we have $F(u) \leq \liminf_{n\to +\infty} F_n(u_n)$.
	\end{itemize}	

	We write sometimes $\Gamma-\lim_{n\to+\infty} F_n = F$.
\end{definition}
Let us recall the following property: \\

\begin{proposition}\label{sum}
Let $(F_n)$ be a  $\Gamma$-convergent sequence, in a Hilbert space $X$, towards a limit $F$ and let $G$ be continuous, then $(F_n+G)$ $\Gamma$-converges to $F+G$.
\end{proposition}
\begin{definition}[$\gamma$-convergence]
	We say that a sequence $(\mu_n)_n$ of measures in $\mathcal{M}_0(D)$ $\gamma$-converge to a measure $\mu$ in $\mathcal{M}_0(D)$ with respect to $F$ (or $(\mu_n)_n$ $\gamma(F)$-converge to $\mu$) if $F_{\mu_n}$ $\Gamma$-converge in $L^2(D)$ to $F_\mu$.
\end{definition}

We give a locality of the $\gamma$-convergence result, then the $\gamma$-compacity of $\mathcal{M}_0(D)$, from \cite{DalMaso1997} and \cite{DalMaso1987} respectively. \\

\begin{proposition}[Locality of the $\gamma$-convergence]
	\label{prop:problem_1:gamma_locality}
	Let $(\mu^1_n)_n$ and $(\mu^2_n)_n$ be two sequences of measures in $\mathcal{M}_0(D)$ which $\gamma(F)$-converge to $\mu^1$ and $\mu^2$ respectively. Assume that $\mu^1_n$ and $\mu^2_n$ coincide q.e. on a subset $D'$ of $D$, for every $n\in\N$. Then $\mu^1$ and $\mu^2$ coincide q.e. on $D'$.
\end{proposition}

\begin{proposition}[$\gamma$-compacity of $\mathcal{M}_0(D)$]
	\label{prop:problem_1:M0:gamma_compacity}
	The set $\mathcal{M}_0(D)$ is compact with respect to the $\gamma$-convergence. Moreover, the class of measures of the form $\infty_{D\setminus A}$, with $A$ open and smooth subset of $D$, is dense in $\mathcal{M}_0(D)$.
\end{proposition}

\begin{note}
	The result above means that, for every $\mu$ in $\mathcal{M}_0(D)$, there exists a family of subsets $(E_n)_n$ of $D$, such that $\infty_{E_n}$ $\gamma$-converge to $\mu$.
\end{note}

We will use in the sequel the shape analysis tools that we introduced in this section, to study the optimization problem \eqref{pb:problem_1_opt_no_constraint}.

\newpage

\section{Proofs for the Analysis of the Model}
\label{appendix:proofs_analysis}

In this appendix are some of the proofs of theorems used in Section \ref{sec:problem_1:continuous_model}.

\begin{proof}[Proof of Theorem \ref{thm:density_of_K_delta}]
	For every $K$ in $\mathcal{K}_\delta(D)$, we have $K\subset D$. Thus $\infty_K$ is in $\mathcal{M}_0^\delta(D)$, and $\mathcal{K}_\delta(D) \subseteq \mathcal{M}_0^\delta(D)$ if we identify a set of $\mathcal{K}_\delta(D)$ by its measure. Thanks to the $\gamma$-compactness, we have $\cl_\gamma \mathcal{K}_\delta(D) \subseteq \mathcal{M}_0^\delta(D)$.\\

	Conversely, let $\mu$ be in $\mathcal{M}_0^\delta(D)$. We need to prove that $\mu$ is in $\cl_\gamma \mathcal{K}_\delta(D)$ i.e. $\mu$ is the $\gamma$-limit of elements of $\mathcal{K}_\delta(D)$. This is to be understood in the sense :  there exist elements $(K_n)_n$ of $\mathcal{K}_\delta(D)$ such that, $\infty_{K_n}$ $\gamma$-converge to $\mu$. Since $\mathcal{M}_0(D)$ is dense and $\mathcal{M}_0^\delta(D)\subset \mathcal{M}_0(D)$, we obtain that there exists a sequence of subsets of $D$, $(K_n)_n$, such that, $\infty_{K_n}$ $\gamma$-converge to $\mu$. We need to show that $K_n$ are in $\mathcal{K}_\delta(D)$. By the localization argument, we can choose $(K_n)_n$ such that $K_n\subseteq (D^{-\delta})^{1/n}$. Making a homothety $\varepsilon_n K_n$, for $\varepsilon_n>0$ such that $\varepsilon_n K_n \subseteq D^{-\delta}$, we have $\varepsilon_n K_n\in \mathcal{K}_\delta(D)$. Moreover, we can choose $\varepsilon_n$ such that $\varepsilon_n \to 1$, therefore $\varepsilon_n K_n$ $\gamma$-converge to $\mu$.
\end{proof}

\begin{proof}[Proof of Theorem \ref{thm:problem_1:convergence_capacity}]
	This proof is similar to the one of Theorem 3.5 in \cite{Belhachmi2009}, we give it for the sake of completeness. Assume that $(\mu_n)_n$ in $\mathcal{M}_0^\delta(D)$ $\gamma$-converges to $\mu$. By $\gamma$-compacity, Proposition \ref{prop:problem_1:M0delta:gamma_compacity}, $\mu$ is in $\mathcal{M}_0^\delta(D)$. Now we prove that $F_{\mu_n}$ $\Gamma$-converges to $F_\mu$ in $L^2(D)$. \\

	\textbullet ~ $\liminf$ : Let $(u_n)_n$ be a sequence in $H^1(D)$ which converge in $L^2(D)$ to $u$. Let $\varphi\in C^\infty_c(D)$, $0\leq\varphi\leq1$ and $\varphi=1$ in $D^{-\delta}$. Then $(u_n\varphi)_n$ is a sequence in $H^1(D^{-\delta})$ and $u_n\varphi \to_{n\to+\infty} u\varphi$ in $L^2(D)$. Since $(\mu_n)_n$ $\gamma(F)$-converges to $\mu$, we have, 

	\[ \liminf_{n\to +\infty} F_{\mu_n}(u_n\varphi) \geq F_\mu(u\varphi) \]

	i.e.

	\[ \liminf_{n\to+\infty}\Big( \alpha\int_D |\nabla (u_n\varphi)|^2\ dx + \int_D (u_n\varphi-f)^2\ d\mu_n \Big) \geq \alpha\int_D |\nabla (u\varphi)|^2\ dx +  \int_D (u\varphi-f)^2\ d\mu. \]

	it follows that,

	\[ \liminf_{n\to+\infty}\Big( \alpha\int_D |\varphi\nabla u_n|^2\ dx + \alpha\int_D |u_n\nabla\varphi|^2\ dx + 2\alpha\int_D u_n\varphi\nabla u_n\cdot \nabla \varphi\ dx + \int_D (u_n\varphi-f)^2\ d\mu_n \Big) \]

	\[ \geq \alpha\int_D |\varphi\nabla u|^2\ dx + \alpha\int_D |u\nabla\varphi|^2\ dx + 2\alpha\int_D u\varphi\nabla u\cdot \nabla \varphi\ dx + \int_D (u\varphi-f)^2\ d\mu. \]
Thus

	\[ \liminf_{n\to+\infty}\Big( \alpha\int_D |\varphi\nabla u_n|^2\ dx  + \int_D (u_n\varphi-f)^2\ d\mu_n \Big) \geq \alpha\int_D |\varphi\nabla u|^2\ dx + \int_D (u\varphi-f)^2\ d\mu. \]

	We have used $0\leq\varphi\leq 1$, andhave taken the sup over $\varphi$,

	\[ \liminf_{n\to+\infty}\Big( \alpha\int_D |\nabla u_n|^2\ dx + \int_D (u_n-f)^2\ d\mu_n \Big) \geq \sup_\varphi \Big( \alpha\int_D |\varphi\nabla u|^2\ dx + \int_D (u\varphi-f)^2\ d\mu \Big). \]

	Since $\mu$ is equals to $0$ in $D\setminus D^{-\delta}$, we have 

	\[ \liminf_{n\to+\infty}\Big( \alpha\int_D |\nabla u_n|^2\ dx + \int_D (u_n-f)^2\ d\mu_n \Big) \geq \sup_\varphi \Big( \alpha\int_D |\nabla u|^2\varphi^2\ dx \Big) + \int_D (u-f)^2\ d\mu. \]

	We get the $\Gamma$-$\liminf$.
	
	\textbullet ~ $\limsup$ : Let $u$ be in $H^1(D)$, such that the principle maximum is fulfilled, i.e. $|u|\leq |f|_\infty$, and $\tilde{u}$ be the extension of $u$ in $H^1_0(D^\delta)$, where $D^\delta$ is the dilatation by a factor $\delta>0$ of $D$. By the locality property of the $\gamma$-convergence, Proposition \ref{prop:problem_1:gamma_locality}, we have that $\mu_n$ $\gamma(G)$-converge to $\mu$ in $D^\delta$, where $G$ is

	\[ G_\mu(u) = \alpha\int_D |\nabla u|^2\ dx + \varepsilon\alpha\int_{D^\delta \setminus D} |\nabla u|^2\ dx + \int_D (u-f)^2\ d\mu, \]

	for $\varepsilon >0$. Hence, there exists a sequence $(u_n^\varepsilon)_n$ of $H^1_0(D^\delta)$ such that $u_n^\varepsilon$ converge to $\tilde{u}$ in $L^2(D^\delta)$ and $G_\mu(\tilde{u}) \geq \limsup_{n\to+\infty} G_{\mu_n}(u_n^\varepsilon)$ i.e.

	\[ \alpha\int_D |\nabla \tilde{u}|^2\ dx + \varepsilon\alpha\int_{D^\delta \setminus D} |\nabla \tilde{u}|^2\ dx + \int_D (\tilde{u}-f)^2\ d\mu \]

	\[ \geq \limsup_{n\to+\infty} \alpha\int_D |\nabla u_n^\varepsilon|^2\ dx + \varepsilon\alpha\int_{D^\delta \setminus D} |\nabla u_n^\varepsilon|^2\ dx + \int_D (u_n^\varepsilon-f)^2\ d\mu_n. \]

	Thus, we have 

	\[ \alpha\int_D |\nabla \tilde{u}|^2\ dx  + \varepsilon\alpha\int_{D^\delta \setminus D} |\nabla \tilde{u}|^2\ dx + \int_D (\tilde{u}-f)^2\ d\mu \]

	\[ \geq \limsup_{n\to+\infty} \alpha\int_D |\nabla u_n^\varepsilon|^2\ dx + \int_D (u_n^\varepsilon-f)^2\ d\mu_n. \]

	Since $\tilde{u}$ is fixed, we make $\varepsilon$ tends to $0$ and extract by a diagonal procedure a subsequence $u_n^{\varepsilon_n}$ converging in $L^2(D^\delta)$ to $\tilde{u}$ i.e.

	\[ \alpha\int_D |\nabla \tilde{u}|^2\ dx +\int_D (\tilde{u}-f)^2\ d\mu \geq \limsup_{n\to+\infty} \alpha\int_D |\nabla u_n^{\varepsilon_n}|^2\ dx  + \int_D (u_n^{\varepsilon_n}-f)^2\ d\mu_n. \]

	Setting $u_n := u_n^{\varepsilon_n}|_{D} \in H^1(D)$, we have since $u = \tilde{u}|_D$,

	\[ \alpha\int_D |\nabla u|^2\ dx  + \int_D (u-f)^2\ d\mu \geq \limsup_{n\to+\infty} \alpha\int_D |\nabla u_n|^2\ dx  + \int_D (u_n-f)^2\ d\mu_n. \]
\end{proof}

\begin{proof}[Proof of Theorem \ref{thm:relaxed_problem}]
	Let $(K_n)_n\subset\mathcal{K}_\delta(D)$ be a maximizing sequence of $E(\infty_K) - \beta\capop_\nu(\infty_K)$ i.e.

	\[ \lim_{n\to+\infty} E(\infty_{K_n}) - \beta\capop_\nu(\infty_{K_n}) = \sup_{K\in\mathcal{K}_\delta(D)} \big( E(\infty_K) - \beta\capop_\nu(\infty_K) \big). \]

	One can extract, since $\mathcal{M}_0^\delta(D)$ is a compact space with respect to the $\gamma$-convergence
	(Proposition \ref{prop:problem_1:M0delta:gamma_compacity}), from $(\infty_{K_n})_n\subset\mathcal{M}_0^\delta(D)$ a $\gamma$-convergent subsequence. We denote by $\mu_\text{lim}$ this $\gamma$-limit. We know that $\mu_\text{lim}$ is in $\mathcal{M}_0^\delta(D)$ since $\mathcal{M}_0^\delta(D)$ is dense with respect to the $\gamma$-convergence (Proposition \ref{prop:problem_1:M0delta:gamma_compacity}). We denote by $G(\mu_\text{lim})$ the value

	\[ G(\mu_\text{lim}) := \lim_{n\to+\infty} E(\infty_{K_n}) - \beta\capop_\nu(\infty_{K_n}). \]

	By definition of the $\gamma$-convergence, we have $\Gamma-\lim_{n\to+\infty} F_{\infty_{K_n}} = F_{\mu_\text{lim}}$. Since $F_{\infty_{K_n}}$ is equicoercive, we can apply Theorem 7.8 in \cite{DalMaso1993}. It leads

	\[ E(\mu_\text{lim}) := \min_{u\in H^1(D)} F_{\mu_\text{lim}}(u) = \lim_{n\to +\infty} \inf_{u\in H^1(D)} F_{\infty_{K_n}}(u) =: \lim_{n\to +\infty} E(\infty_{K_n}). \]

	In addition, by Theorem \ref{thm:problem_1:convergence_capacity} and the uniqueness of the limit, we have 

	\[ \lim_{n\to+\infty} E(\infty_{K_n}) - \beta\capop_\nu(\infty_{K_n}) = G(\mu_\text{lim}) = E(\mu_\text{lim}) - \beta\capop_\nu(\mu_\text{lim}). \]

	Thus, we have



	\begin{align*}
		\sup_{K\in\mathcal{K}_\delta(D)} \big( E(\infty_K) - \beta\capop_\nu(\infty_K) \big) & = \lim_{n\to+\infty} E(\infty_{K_n}) - \beta\capop_\nu(\infty_{K_n}) = E(\mu_\text{lim}) - \beta\capop_\nu(\mu_\text{lim}) \\
		& = \max_{\mu\in\mathcal{M}_0^\delta(D)}\big( E(\mu) - \beta\capop_\nu(\mu) \big).
	\end{align*}
\end{proof}

\newpage

\section{Asymptotic Development Calculus}

Let $x_0$ be in $\R^2$ and $\varepsilon>0$. In this section, we aim to find an estimate of

\[ \int_{B(x_0,\varepsilon)}w\ dx, \]

where $w$ the solution of the problem below : \\

\begin{problem} Find $w$ in $H^1_0\big(B(x_0,\varepsilon)\big)$ such that
	\begin{equation}
		\left\{\begin{array}{rl}
			w - \alpha \Delta w = g, & \text{in}\ B(x_0,\varepsilon), \\
			w = 0, & \text{on}\ \partial B(x_0,\varepsilon).
		\end{array}\right .
	\end{equation}
	\label{pb:appendix:nonhomogeneous:l2}
\end{problem}

We did not find this result in the literature despite it may exist. For the sake of completeness, we propose a way to find this estimate. To solve Problem \ref{pb:appendix:nonhomogeneous:l2}, we use Green functions $G:B(x_0,\varepsilon)\times B(x_0,\varepsilon)$, corresponding to Problem \ref{pb:appendix:nonhomogeneous:l2} which are solution to \\

\begin{problem} Find $G(\cdot,y)$ in $H^1_0\big(B(x_0,\varepsilon)\big)$ such that
    \begin{equation}
    	\left\{\begin{array}{rl}
    		G(x,y) - \alpha \Delta_x G(x,y) = \delta_y(x), & x\in B(x_0,\varepsilon), \\
    		G(x,y) = 0, & x\in\partial B(x_0,\varepsilon),
    	\end{array}\right .
    \end{equation}
    for $y$ in $B(x_0,\varepsilon)$.
    \label{pb:appendix:nonhomogeneous:l2:Green}
\end{problem}

We have, \\

\begin{proposition}
    Let $G$ be Green functions corresponding to Problem \ref{pb:appendix:nonhomogeneous:l2:Green}. Then, for $x$ in $\overline{B(x_0,\varepsilon)}$,
    
    \[ w(x) := \int_{B(x_0,\varepsilon)} g(y)G(x,y)\ dy, \]
    
    is the solution of Problem \ref{pb:appendix:nonhomogeneous:l2}.
    \label{prop:appendix:int_calculus}
\end{proposition}
\begin{proof}
    Let $x$ be in $B(x_0,\varepsilon)$.
    
    \begin{align*}
        w(x) - \alpha \Delta w(x) &= \int_{B(x_0,\varepsilon)} g(y)G(x,y)\ dy - \alpha \int_{B(x_0,\varepsilon)} g(y)\Delta_x G(x,y)\ dy \\
            &= \int_{B(x_0,\varepsilon)} g(y)\big(G(x,y) - \alpha\Delta_x G(x,y)\big)\ dy \\
            &= \int_{B(x_0,\varepsilon)} g(y)\delta_y(x)\ dy \\
            &= g(x).
    \end{align*}
    
    Moreover, we have,  $w(x)=0$, for $x$ on $\partial B(x_0,\varepsilon)$.
\end{proof}

From now, our goal is to find Green functions $G$. To do so, we write $G$ as the sum of a particular solution $G_\text{p}$ of Problem \ref{pb:appendix:nonhomogeneous:l2:Green} without the boundary condition, and the general solution $G_0$ of the homogeneous version of Problem \ref{pb:appendix:nonhomogeneous:l2:Green} such that $G_0 = -G_\text{p}$ on $\partial B(x_0,\varepsilon)$. Below is the main proposition of this section, \\

\begin{proposition} We have when $\varepsilon$ tends to 0,
    \[ \int_{B(x_0,\varepsilon)} w(x)\ dx = -g(x_0)\pi\varepsilon^2\ln(\varepsilon) + O(\varepsilon^2).\]
    \label{prop:appendix:int_calculus:main}
\end{proposition}
\begin{proof}
    For $\varepsilon$ small enough, we have, using Proposition \ref{prop:appendix:int_calculus},
    
    \[ \int_{B(x_0,\varepsilon)} w(x)\ dx = \int_{B(x_0,\varepsilon)} \int_{B(x_0,\varepsilon)} g(y)G(x,y)\ dy\ dx. \]
        
    Using Fubini, we get
    
    \begin{align*}
        \int_{B(x_0,\varepsilon)} w(x)\ dx & = \int_{B(x_0,\varepsilon)} g(y) \int_{B(x_0,\varepsilon)}G(x,y)\ dx\ dy \\
        & = \int_{B(x_0,\varepsilon)} g(y) \int_{B(x_0,\varepsilon)}G_\text{p}(x,y)\ dx\ dy + \int_{B(x_0,\varepsilon)} g(y) \int_{B(x_0,\varepsilon)}G_0(x,y)\ dx\ dy.
    \end{align*}

    Proposition \ref{prop:appendix:int_calculus:1} and Proposition \ref{prop:appendix:int_calculus:2} give us the result.
\end{proof}

It remains to state and to prove Proposition \ref{prop:appendix:int_calculus:1} and Proposition \ref{prop:appendix:int_calculus:2}. We start by giving an explicit expression for $G_\text{p}$ with the following proposition from \cite{Duffy2015}. \\

\begin{proposition} For $x$ and $y$ in $B(x_0,\varepsilon)$ such that $x\neq y$, we have
    \[ G_\text{p}(x,y) = \frac{1}{2\pi}K_0\Big(\frac{1}{\sqrt{\alpha}}|x-y|\Big), \]
    
    where $K_0$ is the modified Bessel function of the second kind, see \cite{Oldham2009}.
\end{proposition}

Then, we compute the first part of Proposition \ref{prop:appendix:int_calculus:main}. \\

\begin{proposition} When $\varepsilon$ tends to $0$, we have,
    \[ \int_{B(x_0,\varepsilon)} g(y) \int_{B(x_0,\varepsilon)}G_\text{p}(x,y)\ dx\ dy = -g(x_0)\frac{\pi}{2}\varepsilon^2\ln(\varepsilon) + O(\varepsilon^2). \]
    \label{prop:appendix:int_calculus:1}
\end{proposition}
\begin{proof} We begin by setting 

    \[ I_\text{p} := \int_{B(x_0,\varepsilon)} g(y) \int_{B(x_0,\varepsilon)}G_\text{p}(x,y)\ dx\ dy. \]

    According to \cite{Oldham2009}, we have the following asymptotic development
    
    \[ K_0(z) = -\ln z + \ln 2 - \gamma + O(z^2|\ln z|), \]
    
    for $z\to 0$, where $\gamma$ denotes the Euler–Mascheroni constant. Then,
    
    \[ G_\text{p}(x,y) = -\frac{1}{2\pi}\Big(\ln|x-y| + \frac{1}{2}\ln \alpha + \ln 2 - \gamma\Big) + O(|x-y|^2|\ln|x-y||), \]
    
    for $|x-y|\to 0$. Thus, 
    
    \begin{align*}
        I_\text{p} &= -\frac{1}{2\pi}\int_{B(x_0,\varepsilon)} g(y) \int_{B(x_0,\varepsilon)}\ln|x-y|\ dx\ dy - \Big(\frac{1}{2}\ln \alpha + \ln 2 - \gamma\Big)\frac{\varepsilon^2}{2}\int_{B(x_0,\varepsilon)} g(y) \ dy \\
        & \hspace{2cm} + O(1)\int_{B(x_0,\varepsilon)} g(y) \int_{B(x_0,\varepsilon)}|x-y|^2|\ln|x-y||\ dx\ dy.
    \end{align*}
    
    Using Taylor's formula, we have 
    
    \begin{align*}
        I_\text{p} &= -\frac{1}{2\pi}g(x_0)\int_{B(x_0,\varepsilon)}  \int_{B(x_0,\varepsilon)}\ln|x-y|\ dx\ dy + O(1)\int_{B(x_0,\varepsilon)} \|y-x_0\| \int_{B(x_0,\varepsilon)}\ln|x-y|\ dx\ dy \\
        & \hspace{2cm} - \Big(\frac{1}{2}\ln \alpha + \ln 2 - \gamma\Big)\frac{\pi}{2}g(x_0)\varepsilon^4 + O(\varepsilon^4) + O(1)\int_{B(x_0,\varepsilon)}\int_{B(x_0,\varepsilon)}|x-y|^2|\ln|x-y||\ dx\ dy \\
        & \hspace{2cm} + O(1)\int_{B(x_0,\varepsilon)} \|y-x_0\| \int_{B(x_0,\varepsilon)}|x-y|^2|\ln|x-y||\ dx\ dy.
    \end{align*}
    
    By setting $y = \varepsilon\tilde{y} + x_0$ and $x = \varepsilon\tilde{x} + x_0$, we have \[ \int_{B(x_0,\varepsilon)}  \int_{B(x_0,\varepsilon)}\ln|x-y|\ dx\ dy = \varepsilon^2\int_{B(0,1)}  \int_{B(0,1)}\ln(\varepsilon|\tilde{x}-\tilde{y}|)\ d\tilde{x}\ d\tilde{y} = \pi^2\varepsilon^2\ln\varepsilon + O(\varepsilon^2). \]
    
    Using the same substitution for the remaining integrals, we got the result.
\end{proof}

And we finish by computing the second part of Proposition \ref{prop:appendix:int_calculus:main}. \\

\begin{proposition} When $\varepsilon$ tends to $0$, we have,
    \[ \int_{B(x_0,\varepsilon)} g(y) \int_{B(x_0,\varepsilon)}G_0(x,y)\ dx\ dy = O(\varepsilon^2). \]
    \label{prop:appendix:int_calculus:2}
\end{proposition}
\begin{proof} We set

    \[ I_0 := \int_{B(x_0,\varepsilon)} g(y) \int_{B(x_0,\varepsilon)}G_0(x,y)\ dx\ dy. \]
    
    Using Taylor's formula on $g$ around $x_0$,

    \begin{align*}
        I_0 &= g(x_0) \int_{B(x_0,\varepsilon)}\int_{B(x_0,\varepsilon)}G_0(x,y)\ dx\ dy + O(1)\int_{B(x_0,\varepsilon)}\|y-x_0\| \int_{B(x_0,\varepsilon)}G_0(x,y)\ dx\ dy \\
        & \leq g(x_0) \pi\varepsilon^2 \int_{B(x_0,\varepsilon)}\|G_0(\cdot,y)\|_{L^\infty\big(B(x_0,\varepsilon)\big)}\ dy + O(\varepsilon^2)\int_{B(x_0,\varepsilon)}\|y-x_0\| \|G_0(\cdot,y)\|_{L^\infty\big(B(x_0,\varepsilon)\big)}\ dy.
    \end{align*}
    
    Since $G_0$ satisfies the maximum principle \cite{Evans2010},
    
    \begin{align*}
        I_0 &\leq g(x_0) \pi\varepsilon^2 \int_{B(x_0,\varepsilon)}\|G_0(\cdot,y)\|_{L^\infty\big(\partial B(x_0,\varepsilon)\big)}\ dy + O(\varepsilon^2)\int_{B(x_0,\varepsilon)}\|y-x_0\| \|G_0(\cdot,y)\|_{L^\infty\big(\partial B(x_0,\varepsilon)\big)}\ dy \\
        &= g(x_0) \pi\varepsilon^2 \int_{B(x_0,\varepsilon)}\|G_\text{p}(\cdot,y)\|_{L^\infty\big(\partial B(x_0,\varepsilon)\big)}\ dy + O(\varepsilon^2)\int_{B(x_0,\varepsilon)}\|y-x_0\| \|G_\text{p}(\cdot,y)\|_{L^\infty\big(\partial B(x_0,\varepsilon)\big)}\ dy \\
        &= g(x_0) \frac{1}{2}\varepsilon^2 \int_{B(x_0,\varepsilon)}\Big\|K_0\Big(\frac{1}{\sqrt{\alpha}}|\cdot-y|\Big)\Big\|_{L^\infty\big(\partial B(x_0,\varepsilon)\big)}\ dy + O(\varepsilon^2)\int_{B(x_0,\varepsilon)}\|y-x_0\| \Big\|K_0\Big(\frac{1}{\sqrt{\alpha}}|\cdot-y|\Big)\Big\|_{L^\infty\big(\partial B(x_0,\varepsilon)\big)}\ dy.
    \end{align*}

    According to \cite{Oldham2009}, $K_0$ is an increasing function. Thus, for $y$ in $B(x_0,\varepsilon)$,
    
    \[ \Big\|K_0\Big(\frac{1}{\sqrt{\alpha}}|\cdot-y|\Big)\Big\|_{L^\infty\big(\partial B(x_0,\varepsilon)\big)} := \sup_{x\in\partial B(x_0,\varepsilon)} \Big|K_0\Big(\frac{1}{\sqrt{\alpha}}|x-y|\Big)\Big|, \]
    
    is attained where $|x-y| := \sqrt{r_x^2 + r_y^2 - 2r_x r_y\cos{(\theta_x-\theta_y)}}$ is maximal, i.e. when $\cos{(\theta_x-\theta_y)} = -1$, i.e. for $\theta_x=\pi+\theta_y$. In that case,
    
    \[ |x-y| = \sqrt{\varepsilon^2 + r_y^2 + 2\varepsilon r_y} = \varepsilon + r_y. \]
    
    Thus,
    
    \[ \|G_\text{p}(\cdot,y)\|_{L^\infty\big(\partial B(x_0,\varepsilon)\big)} = \frac{1}{2\pi} K_0\Big(\frac{1}{\sqrt{\alpha}}(\varepsilon+r_y)\Big). \]
    
    Then, we have
    
    \begin{align*}
        I_0 & \leq g(x_0) \pi \varepsilon^2 \int_0^\varepsilon r_y K_0\Big(\frac{1}{\sqrt{\alpha}}(\varepsilon+r_y)\Big)\ dy + O(\varepsilon^2)\int_0^\varepsilon r_y^2 K_0\Big(\frac{1}{\sqrt{\alpha}}(\varepsilon+r_y)\Big)\ dy \\
        & \leq g(x_0) \frac{\pi}{2}\varepsilon^4 K_0\Big(\frac{2\varepsilon}{\sqrt{\alpha}}\Big) + O(\varepsilon^5)K_0\Big(\frac{2\varepsilon}{\sqrt{\alpha}}\Big).
    \end{align*}

    Again, we use that, when $z$ tends to $0$,
    
    \[ K_0(z) = -\ln z + \ln 2 - \gamma + O(z^2|\ln z|), \]
    
    and get, since $\varepsilon$ tends to $0$,
    
    \[ K_0\Big(\frac{2\varepsilon}{\sqrt{\alpha}}\Big) = -\ln \varepsilon + \frac{1}{2}\ln \alpha - \gamma + O(\varepsilon^2|\ln \varepsilon|). \]
    
    Therefore,

    \begin{align*}
        I_0 &\leq -g(x_0) \frac{\pi}{2}\varepsilon^4\ln \varepsilon + g(x_0) \frac{\pi}{4}\varepsilon^4\ln \alpha - g(x_0) \frac{\pi}{2}\varepsilon^4\gamma + O(\varepsilon^5\ln \varepsilon) \\
        &= O(\varepsilon^2).
    \end{align*}
\end{proof}
\newpage\section{Proof of the Estimate of \texorpdfstring{$\theta$}{theta}}
\label{appendix:theta}

We aim to give some estimate of the function $\theta$ defined in Theorem \ref{thm:g-convergence}. In the sequel, we denote by $v_K$ the solution of Problem \ref{pb:problem_1:v}, $g := \alpha\Delta f$ and $t_1 := \sqrt{2}/2$. We will widely use the following maximum principle (See \cite{Evans2010} Theorem 2 in Section 6.4). \\

\begin{theorem}[Weak maximum principle] We assume that $v_K$ is in $C^2(D)\cap C^0(\bar{D})$.
	If $g\geq 0$ in $D\setminus K$, then $v_K\geq 0$ in $D$.
	\label{prop:problem_1:v:maxprinciple}
\end{theorem}

Moreover, we will need the following properties : \\

\begin{lemma} We have 
	\[ \|v_K\|_{L^2(D)} \leq \big(1+\alpha C(D)\big)^{-1}\|g\|_{L^2(D)} \]
	and

	\[ \|v_K\|_{L^1(D)} \leq |D|^{1/2}(1+\alpha C(D)\big)^{-1} \|g\|_{L^2(D)}. \]
	\label{prop:problem_1:solution_bounded}
\end{lemma}
\begin{proof}
    Use the weak formulation, Poincaré inequality and Hölder inequality.
\end{proof}

\begin{note}

	If $D\subset\R^2$ is convex we have, according to \cite{Payne1960},

	\[ \|v_K\|_{L^2(D)} \leq \big(1+\alpha\pi^2 \diam(D)^{-2}\big)^{-1}\|g\|_{L^2(D)} . \]

	\[ \|v_K\|_{L^1(D)} \leq |D|^{1/2}(1+\alpha\pi^2 \diam(D)^{-2}\big)^{-1} \|g\|_{L^2(D)}. \]

	Moreover, if $g=1$, we have $\|v_K\|_{L^1(D)} \leq (1+\alpha\pi^2 \diam(D)^{-2}\big)^{-1}|D|$.
\end{note}\vspace{0.2cm}

\begin{lemma}
	We have, for $m$ in $(0,t_1)$, \[ \theta(m) \leq C_1(\alpha)\ln{m^{-1}} + C_2(\alpha), \]
	where $C_1$ and $C_2$ are constants depending on $\alpha$.
	\label{prop:problem_1:upper_bound}
\end{lemma}
\begin{proof}
	We consider a particular family of sets $K_n$ in $\mathcal{A}_{m,n}$. We choose an integer $k$ such that $n = k^2$ and we suppose $K_n\in\mathcal{A}_{m,n}$ are composed of $n$ balls of radius $m/k$, with their centers superposing the centers of the $k^2$ squares of side $1/k$ of a regular lattice partitioning the square $I^2$.

	\begin{figure}[!ht]
		\centering
		\subfloat[$I^2\setminus K_1$]{
			\includegraphics[width=3.5cm]{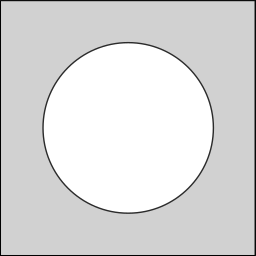}
		}
		\qquad
	    \subfloat[$I^2\setminus K_4$]{
			\includegraphics[width=3.5cm]{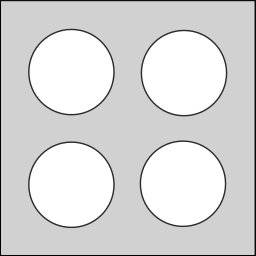}
		}
		\qquad
		\subfloat[$I^2\setminus K_9$]{
			\includegraphics[width=3.5cm]{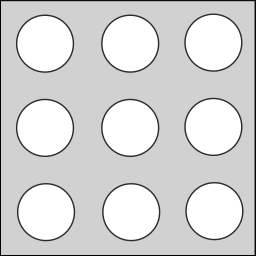}
		}
		\caption{Drawing of $I^2\setminus K_n$ with $n := k^2$, for $k=1,2,3$.}
	\end{figure} 

	Let us denote $v^1_{K_n}$ the solution of Problem \ref{pb:problem_1:v} with $g=1$, $K=K_n$ and $D=I^2$. It holds $\int_I v^1_{K_1}\ dx = n\int_I v^1_{K_n}\ dx$. We recall 

	\[ \theta(m) := \inf_{K_n\in\mathcal{A}_{m,n}} \liminf_n n \int_D gv_{K_n}\ dx. \]

	In particular when $g=1$ \[ \theta(m) \leq \liminf_n n \int_D v^1_{K_n}\ dx = \int_I v^1_{K_1}\ dx. \]

	We have $I^2\subset B(x_0,t_1)$. Therefore if we denote by $w$ the solution of Problem \ref{pb:problem_1:v} with $g=1$, $D=B(x_0,t_1)$ and $K=K_1:=B(x_0,m)$, it holds by the maximum principle that $v^1_{K_1} \leq w$. Then we have the following estimate 

	\[ \theta(m) \leq \int_{B(x_0,t_1)} w\ dx. \]

	Let us consider the following problem

	\begin{equation}
		\left\{\begin{array}{rl}
			-\alpha\Delta \tilde{w} = 1 & \text{in}\ B(x_0,t_1)\setminus\overline{B(x_0,m)}, \\
			\tilde{w} = 0 & \text{in}\ \overline{B(x_0,m)}, \\
			\frac{\partial \tilde{w}}{\partial \mathbf{n}} = 0 & \text{on}\ \partial B(x_0,t_1).
		\end{array}\right .
	\end{equation}

	Now, we set $e:=w-\tilde{w}$. Thus we have for all $v$ in $H^1_0(B(x_0,t_1)\setminus\overline{B(x_0,m)})$

	\[ \alpha\int_{B}\nabla e\cdot\nabla v\ dx = \alpha\int_{B}\nabla w\cdot\nabla v\ dx - \alpha\int_{B}\nabla \tilde{w}\cdot\nabla v\ dx = -\int_{B}wv\ dx, \]

	i.e. $e$ satisfies the problem below

	\begin{equation}
		\left\{\begin{array}{rl}
			-\alpha\Delta e = -w & \text{in}\ B(x_0,t_1)\setminus\overline{B(x_0,m)}, \\
			e = 0 & \text{in}\ \overline{B(x_0,m)}, \\
			\frac{\partial e}{\partial \mathbf{n}} = 0 & \text{on}\ \partial B(x_0,t_1).
		\end{array}\right .
	\end{equation}

	Since $w\geq 0$, we have by the maximum principle that $e\leq 0$ i.e. $0\leq w\leq \tilde{w}$. By consequence

	\[ \theta(m) \leq \int_{B(x_0,t_1)} \tilde{w}\ dx. \]

	The solution $\tilde{w}$ have been computed in \cite{Buttazzo2006}. Due to the radial symmetry of $\tilde{w}$, we can get explicitly $\tilde{w}$, solution of :

	\begin{equation}
		\left\{\begin{array}{rl}
			\tilde{w}''(r) + \frac{1}{r} \tilde{w}'(r) = -\frac{1}{\alpha} & \text{if}\ m<r< t_1, \\
			\tilde{w} = 0 & \text{if}\ 0\leq r\leq m, \\
			\tilde{w}'(t_1) = 0 & .
		\end{array}\right .
	\end{equation}

	For $r=|x-x_0|$ we have

	\[ \tilde{w}(x) = \begin{cases}
		k\ln(\frac{r}{m}) - \frac{1}{4\alpha}(r^2-m^2) & \text{if}\ m<r< t_1 \\ 
		0 & \text{if}\ 0\leq r\leq m \\

	\end{cases},\ k=\frac{m t_1^2}{2\alpha}. \]

	Integrating $\tilde{w}$ over $B(x_0,t_1)$

	\begin{align*}	
		\int_{B(x_0,t_1)} \tilde{w}\ dx & = 2\pi\int_m^{t_1} \Big(k\ln(\frac{r}{m}) - \frac{1}{4\alpha}(r^2-m^2)\Big)r\ dr \\
		& = 2\pi k\int_m^{t_1} r\ln(\frac{r}{m})\ dr - \frac{\pi}{2\alpha}\int_m^{t_1} r^3\ dr + \frac{\pi}{2\alpha}m^2\int_m^{t_1} r\ dr \\
		& = \pi kt_1^2\ln{\frac{t_1}{m}} + \frac{\pi}{2}\underbrace{(\frac{m^2}{2\alpha} - \frac{k}{m})}_{\leq\ 0}(t_1^2 - m^2) + \frac{\pi}{8\alpha}\underbrace{(m^4-t_1^4)}_{\leq\ 0} \\
		& \leq \frac{\pi t_1^4}{2\alpha}m \ln{\frac{t_1}{m}} \leq \frac{\pi t_1^4}{2\alpha}\ln{\frac{t_1}{m}} = \underbrace{\frac{\pi t_1^4}{2\alpha}}_{=:\ C_1(\alpha)}\ln{m^{-1}} + \underbrace{\frac{\pi t_1^4}{2\alpha}\ln{t_1}}_{=:\ C_2(\alpha)}.
	\end{align*}
\end{proof} \vspace{0.2cm}

\begin{lemma}
	We have, for $m$ in $(0,t_1)$, \[ C_1(\alpha)\ln(m^{-1}) - C_2(\alpha) \leq \theta(m),\]

	where $C_1$ and $C_2$ are constants depending on $\alpha$.

	\label{prop:problem_1:lower_bound}
\end{lemma}
\begin{proof}
	We fix $n\in\N$ and $(x_i)_{i=1,\dots, n}\in I^2$. We consider the sets $K_m := \bigcup_{i=1}^n \overline{B(x_i, m n^{-1/2})}$ and $U_m := I^2\setminus K_m$. Let us denote $u_m^1$ the solution of Problem \ref{pb:problem_1:v} with $g=1$, $K=K_m$ and $D=I^2$. Using Holder inequality we have

	\[ \Big(\int_{\partial U_m}  \frac{\partial u_m^1}{\partial n}d\mathcal{H}^1\Big)^2 \leq \int_{\partial U_m} \Big|\frac{\partial u_m^1}{\partial n}\Big|^2\ d\mathcal{H}^1 \times \int_{\partial U_m}  \ d\mathcal{H}^1 = \mathcal{H}^1(\partial U_m) \int_{\partial U_m} \Big|\frac{\partial u_m^1}{\partial n}\Big|^2\ d\mathcal{H}^1. \]

	Moreover, we have thanks to the Green formula

	\[ \int_{\partial U_m} \frac{\partial u_m^1}{\partial n}d\mathcal{H}^1 = \int_{U_m}  \Delta u_m^1\ dx = \frac{1}{\alpha}\int_{U_m}  (u_m^1-1)\ dx = \frac{1}{\alpha}\|u_m^1\|_{L^1(U_m)} - \frac{1}{\alpha}|U_m|. \]

	Thus 
	
	\begin{align*}
	    & \Big(\frac{1}{\alpha}\|u_m^1\|_{L^1(U_m)} - \frac{1}{\alpha}|U_m|\Big)^2 \leq \mathcal{H}^1(\partial U_m) \int_{\partial U_m} \Big|\frac{\partial u_m^1}{\partial n}\Big|^2\ d\mathcal{H}^1 \\
	    \Leftrightarrow & \frac{1}{\alpha^2}|U_m|^2 - \frac{2}{\alpha^2}\|u_m^1\|_{L^1(U_m)}|U_m| \leq \mathcal{H}^1(\partial U_m) \int_{\partial U_m} \Big|\frac{\partial u_m^1}{\partial n}\Big|^2\ d\mathcal{H}^1.
	\end{align*}

	The fact that $|U_m| \geq 1-2\pi m^2$ and Property \ref{prop:problem_1:solution_bounded} give us 

	\[ \frac{2\pi^2\alpha - 1}{\alpha^2(1+2\pi^2\alpha)} - \frac{2\pi}{\alpha^2}m \leq \mathcal{H}^1(\partial U_m) \int_{\partial U_m} \Big|\frac{\partial u_m^1}{\partial n}\Big|^2\ d\mathcal{H}^1. \]

	Also, it holds $\mathcal{H}^1(\partial U_m) \leq 2\pi m\sqrt{n}$. Then 
	
	\begin{align*}
	    & \frac{2\pi^2\alpha - 1}{\alpha^2(1+2\pi^2\alpha)} - \frac{2\pi}{\alpha^2}m \leq 2\pi m\sqrt{n}\int_{\partial U_m} \Big|\frac{\partial u_m^1}{\partial n}\Big|^2\ d\mathcal{H}^1 \\
	    \Leftrightarrow & \frac{2\pi^2\alpha - 1}{2\pi\alpha^2(1+2\pi^2\alpha)}\frac{1}{m} - \frac{1}{\alpha^2} \leq \sqrt{n}\int_{\partial U_m} \Big|\frac{\partial u_m^1}{\partial n}\Big|^2\ d\mathcal{H}^1.
	\end{align*}

	Using that $-\frac{d F}{dm} = \sqrt{n}^{-1} \int_{\partial U_m} \Big|\frac{\partial u_m^1}{\partial n}\Big|^2\ d\mathcal{H}^1$, we have

	\[ \frac{2\pi^2\alpha - 1}{2\pi\alpha^2(1+2\pi^2\alpha)}\frac{1}{m} - \frac{1}{\alpha^2} \leq - n \frac{d F}{dm}. \]

	Integrating over $[m_1,m_2] \subset (0,t_1)$ yields to

	\[ \frac{2\pi^2\alpha - 1}{2\pi\alpha^2(1+2\pi^2\alpha)}\ln\Big(\frac{m_2}{m_1}\Big) - \frac{m_2-m_1}{\alpha^2} + nF_{m_2} \leq nF_{m_1}. \]

	Taking inf over $x_i$ and passing to $\liminf$ over $n$ when $n$ tends to $+\infty$ leads to

	\[ \frac{2\pi^2\alpha - 1}{2\pi\alpha^2(1+2\pi^2\alpha)}\ln\Big(\frac{m_2}{m_1}\Big) - \frac{m_2-m_1}{\alpha^2} + \theta(m_2) \leq \theta(m_1). \]

	In particular, if $m_2 = t_1$ and $m_1 = m$, $0<m<t_1 = \frac{\sqrt{2}}{2}$,
	
	\begin{align*}
	    & \frac{2\pi^2\alpha - 1}{2\pi\alpha^2(1+2\pi^2\alpha)}\ln\Big(\frac{t_1}{m}\Big) - \frac{t_1-m}{\alpha^2} \leq \theta(m) \\
	    \Leftrightarrow & \frac{2\pi^2\alpha - 1}{2\pi\alpha^2(1+2\pi^2\alpha)}\ln\Big(\frac{t_1}{m}\Big) - \frac{t_1}{\alpha^2} \leq \theta(m) \\
	    \Leftrightarrow & \underbrace{\frac{2\pi^2\alpha - 1}{2\pi\alpha^2(1+2\pi^2\alpha)}}_{=:\ C_1(\alpha)}\ln(m^{-1}) - \underbrace{\Big(\frac{2\pi^2\alpha - 1}{2\pi\alpha^2(1+2\pi^2\alpha)}\ln(t_1^{-1}) + \frac{t_1}{\alpha^2}\Big)}_{=:\ C_2(\alpha)} \leq \theta(m).
	\end{align*}
\end{proof}

\end{document}